\documentclass[11pt]{amsart}%
\usepackage{amsfonts,amssymb,graphicx}
\usepackage{epsfig,amsmath,latexsym}
\usepackage{amscd, amssymb, amsthm, amsmath,eucal, verbatim}
\usepackage{epstopdf}
\usepackage{amsmath}
\usepackage{amsfonts}
\usepackage{amssymb}
\usepackage{graphicx}%
\providecommand{\U}[1]{\protect\rule{.1in}{.1in}}
\newtheorem{theorem}{Theorem}[section]

\newtheorem{claim}[theorem]{Claim}

\newtheorem{corollary}[theorem]{Corollary}

\newtheorem{definition}[theorem]{Definition}
\newtheorem{example}[theorem]{Example}

\newtheorem{lemma}[theorem]{Lemma}

\newtheorem{proposition}[theorem]{Proposition}
\newtheorem{remark}[theorem]{Remark}

\begin{document}
\title{Simplicial Energy and Simplicial Harmonic Maps}
\author{Joel Hass}
\address{Department of Mathematics, University of California, Davis, California 95616}
\email{hass@math.ucdavis.edu}
\author{Peter Scott}
\address{Department of Mathematics, University of Michigan, Ann Arbor, Michigan 48109}
\email{pscott@umich.edu}
\thanks{Partially supported by NSF grants}
\date{\today }

\begin{abstract}
We introduce a combinatorial energy for maps of triangulated surfaces with
simplicial metrics and analyze the existence and uniqueness properties of the
corresponding harmonic maps. We show that some important applications of
smooth harmonic maps can be obtained in this setting.

\end{abstract}
\maketitle

\section{Introduction}

\label{intro}

The energy of a map from one manifold to another is a measure of the total
stretching of the map. Energy minimizing harmonic maps have found numerous
applications in geometry, analysis, algebra and topology. The two survey
papers of Eells and Lemaire \cite{EellsLemaire2, EellsLemaire3} give an
introduction to the extensive literature of this subject.

One of the fundamental existence theorems states that a nontrivial map of a
Riemannian manifold $F$ to a negatively curved manifold $M$ is homotopic to a
unique harmonic map. When $F$ is a surface, each conformal class of metrics on
$F$ gives rise to the same map, and smooth families of domains give rise to
smooth families of harmonic maps \cite{EellsSampson, Hartman}.

In this paper we introduce a new type of energy that leads to what we call a
\emph{simplicial harmonic map}. The key idea is to give a new, more
combinatorial, definition of area, and then find a corresponding definition
for energy. The payoff is that the existence and regularity of simplicial
harmonic maps are simple to prove, but they retain enough of the features of
smooth harmonic maps to be useful in applications. In this paper we focus on
maps from surfaces into non-positively curved manifolds and spaces. We briefly
discuss extensions to maps between manifolds of any dimension.

One motivation for studying energy from a more combinatorial point of view is
to obtain numerical methods for computing minimal and harmonic surfaces. This
was carried out in the work of Pinkall and Polthier
\cite{PinkallPolthier:1993}. See \cite{Polthier:2002} for a survey of related
work. The discrete harmonic maps of Pinkall and Polthier are maps of a
triangulation of a domain surface into ${\mathbb{R}}^{n}$ that are linear on
each $2$--simplex. Their formula gives the natural energy for a
piecewise-linear approximation of a smooth map, and has found many uses in
computational geometry \cite{Desbrun:1999}, where harmonic maps give a
preferred choice of surface map for purposes such as texturing and meshing.
While the Pinkall-Polthier harmonic maps retain many of the useful features of
smooth harmonic maps, particularly when the $2$--simplices in the domain are
acute angled, they do not in general satisfy the convex hull or mean value
properties. (See section~\ref{meanvalueandconvexhullproperties} for a
discussion of these properties and their consequences.) An example of Polthier
and Rossman \cite{PolthierRossman:2002} shows that a discrete harmonic map may
fail to satisfy the convex hull property, with an interactive demonstration
available at \cite{PolthierRossman:2000}.

The simplicial harmonic maps we describe in this paper are defined for very
general target spaces and do satisfy a convex hull property. A consequence is
that if $f$ is a simplicial harmonic immersion of a surface $F$ into a
Riemannian manifold $M$, then the induced curvature on $F$ is more negative
than the sectional curvature of $M$. When $M$ itself is negatively curved, the
Gauss-Bonnet theorem then implies that the area of $f$ is bounded above by a
constant times the Euler characteristic of $F$. When mapped into Euclidean
space, simplicial harmonic maps also satisfy a mean value property and the
maximum principle.

In this paper, we use the completeness and non-positive curvature of a
Riemannian manifold $M$ to conclude that each homotopy class of arcs in $M$
contains a unique geodesic representative (rel boundary). There are other
conditions that imply this property, and our construction then immediately
applies. For example, Teichmuller space with either the Teichmuller metric or
the Weil-Peterson metric has a unique geodesic connecting any two points
\cite{Wolpert:87}, although the Weil-Peterson metric is not complete.
Similarly, any two points in the interior of a hemisphere are connected by a
unique geodesic arc. More generally, our methods work equally well if $M$ is a
metrized graph, a $CAT(0)$ space, or a path metric space with the property
that any two points are joined by a unique geodesic in each homotopy class.

This paper is organized as follows. Initially we consider maps of triangulated
compact surfaces into manifolds. In Section \ref{simplicialarea} we define the
simplicial area of a map, and in Section \ref{simplicialenergy} we define
simplicial energy. In Section \ref{polygonaldecompositions} we briefly discuss
how to define simplicial area and energy of maps when the source surface is
divided into polygons which need not be triangles. In Section
\ref{simplicialharmonic} we define simplicial harmonic maps. In Section
\ref{meanvalueandconvexhullproperties} we establish the mean value and convex
hull properties of simplicial harmonic maps.

In Section~\ref{families} we consider the problem of deforming families of
surfaces in $3$--manifolds to surfaces of small area. Given any surface in a
$3$--manifold, it is possible to find a homotopy that deforms the surface to a
collection of minimal surfaces joined by thin tubes. In a hyperbolic
$3$--manifold, there is a universal upper bound for the area of such a surface
in terms of its Euler characteristic $\chi$; a surface with $\chi<0$ can be
homotoped to have area at most $2\pi\left\vert \chi\right\vert $. But there is
no canonical homotopy to such a small area surface. For many applications one
would like to start with a family of surfaces and continuously homotope the
entire family of surfaces so that each surface in the family has small area.
Area deformation techniques based on mean curvature and related flows are not
well suited to such a process, due to the formation of singularities and the
non-uniqueness of least area surfaces in a homotopy class. In
\cite{HassThompsonThurston} it was shown that smooth energy deformation
techniques based on analytic results of Eells and Sampson \cite{EellsSampson}
can be applied successfully in this context. We show that a simplified theory
based on simplicial energy suffices to give a deformation of a family of
surfaces to a family of small area surfaces. Then we extend this discussion to
explain how the smooth energy deformation techniques in
\cite{HassThompsonThurston} can be replaced by simpler simplicial techniques.

In Section \ref{genus} we introduce the idea of the genus of an $n$%
--dimensional manifold, extending the notion of Heegaard genus from dimension
three to all dimensions. We find connections between this $n$--dimensional
genus and the areas of surfaces in a sweepout of an $n$--manifold.

The simplicial energy of a map depends on the choice of a simplicial metric on
its domain. In Section \ref{limits} we show the existence of maps which are
global minimizers of simplicial energy over the space of all simplicial
metrics. In Section ~\ref{diffeo} we consider the problem of finding canonical
parametrizations, or canonical triangulations of a surface. We solve this by
showing that if a simplicial harmonic map from a closed surface of negative
Euler characteristic to a non-positively curved Riemannian surface is
homotopic to a homeomorphism, then it is itself a homeomorphism. Then we give
applications of this result to study the space of straight triangulations of a
compact surface $F$ with a Riemannian metric of non-positive curvature. (A
triangulation of $F$ is called straight if each edge is a geodesic arc.) We
show that the space of straight triangulations of $F$ of a fixed combinatorial
type is connected. We note that in the case when $F$ is the disk with a flat
metric, Bloch, Connelly and Henderson \cite{BlochConnellyHenderson} proved
much more, showing that this space is homeomorphic to some Euclidean space and
so contractible. We thank Igor Rivin for telling us about this work, and for
suggesting that our ideas could be useful for studying the space of straight
triangulations of more general surfaces. Finally, in Section
\ref{higherdimensions} we briefly discuss generalizations to maps of higher
dimensional manifolds.

\section{Simplicial area for surfaces}

\label{simplicialarea}

In this section we define the simplicial area of a compact $2$--dimensional
surface $F$. In place of a Riemannian metric on $F$, we specify a
triangulation $\tau$ and a map $l$ that assigns to each edge $e_{i}$, $1\leq
i\leq r$, of $\tau$ a length $l_{i}=l(e_{i})>0$, with the lengths $l_{i}$
satisfying the triangle inequality for each triangle of $\tau$. We call such
an assignment a \emph{simplicial metric} on $F$, and denote it by $(F,\tau
,l)$, or just $l$ when the context is clear. Note that we allow the
possibility that for some triangles one of the triangle inequalities is an
equality. We do not require the triangulation $\tau$ to be combinatorial. Thus
we allow triangulations where two, or even all three, vertices of a single
triangle coincide in $F$, or where a pair of triangles intersect in two or
three vertices. But we assume that $\tau$ has the structure of a $\Delta
$--complex \cite{Hatcher}. This is not a serious restriction. This means that
$F$ is constructed from a disjoint union of triangles, each of whose vertices
is ordered, by identifying faces of the triangles in an order preserving way.
Thus the given orderings can be extended to an ordering of all the vertices of
$\tau$. For simplicity we will always discuss triangles in $\Delta$--complexes
as if they were embedded in $F$. If a triangle is not embedded in $F$, one
needs to consider the source triangle in the initial disjoint union of
triangles. Each simplex in a $\Delta$--complex has a minimal (in the given
order) vertex, which we call its preferred vertex. When considering limits of
simplicial metrics, it is convenient to also allow edges in the triangulation
that are assigned zero length. We call such an extension a \emph{simplicial
quasi-metric} on $F$.

We wish to assign a meaning to the area of $(F,\tau,l)$. A natural choice for
the area of a triangle $T$ with edge lengths $a,b,c$ is the area of the
Euclidean triangle with the same edge lengths. This is given by the classical
Heron's formula, see \cite{Heron} for example,
\[
A(T)=\sqrt{\frac{(a+b+c)(a+b-c)(b+c-a)(a+c-b)}{16}}.
\]
We will however use an alternate formula, one that is much simpler, but still
useful. We define the \emph{simplicial area} $A_{S}$ of the triangle $T$ by
the formula%
\[
A_{S}(T)=ab+bc+ca.
\]

We note that the Euclidean area $A$ of a triangle is at most $xy/2$, for any
pair $x$, $y$ of edge lengths. So $A\leq ab/2$, $A\leq bc/2$ and $A\leq ca/2$.
Equality occurs when an appropriate angle is a right angle. Adding these
inequalities yields the inequality%
\[
A<A_{S}/6.
\]

One defect of our definition of the simplicial area of a triangle, is that if
we subdivide a triangle into smaller triangles, the simplicial area is not, in
general, additive. There is however a special case where additivity holds.
This is when we add a vertex at the midpoint of each edge of a Euclidean
triangle $T,$ and subdivide $T$ into four similar triangles. We call this
\textit{conformal subdivision}.

Another defect of our definition is that the simplicial area of a triangle can
only be zero when $a=b=c=0$. For some purposes, this is not a problem. In the
cases where this is a serious problem, we use quadrilaterals rather than
triangles. See section \ref{polygonaldecompositions}.

\begin{definition}
\label{defnofsimplicialarea}The \emph{simplicial area} of $(F,\tau,l)$ is
\[
A_{S}(F,\tau,l)=\sum_{T\in\tau}A_{S}(T)=\sum_{T\in\tau}l_{i}l_{j}+l_{j}%
l_{k}+l_{k}l_{i}%
\]
where $l_{i},l_{j},l_{k}$ are the lengths of the edges of the $2$--simplex $T$
of $\tau$.
\end{definition}

Unlike Heron's formula, this formula makes sense even when the triangle
inequality does not hold. However the triangle inequality will prove useful,
so we have built it into the definition of a simplicial metric. We will see
that simplicial area retains enough similarities to the standard notion of
area to be useful in many applications.

We now use this idea to define a simplicial analogue of the area of a map. Let
$f:F\rightarrow M$ be a piecewise-smooth map of a metrized triangulated
compact surface $(F,\tau,l)$ to a Riemannian manifold $M$ and let $L_{i}$
denote the length of $f|e_{i}$. Of course, this can only be defined when
$f|e_{i}$ is rectifiable.

\begin{definition}
The simplicial area of a map $f:(F,\tau,l)\rightarrow M$ is
\[
A_{S}(f)=\sum_{T\in\tau}L_{i}L_{j}+L_{j}L_{k}+L_{k}L_{i},
\]
where the sum is taken over all $2$--simplices of $\tau$.
\end{definition}

The simplicial area of $f$ does not depend on the simplicial metric $l$, but
it does depend on the triangulation of $F$.

\section{Simplicial energy}

\label{simplicialenergy}

The energy of a smooth map $f: (F,g) \to(M,h)$ of a Riemannian surface $(F,g)$
to a Riemannian manifold $(M,h)$ is
\[
E(f) = \frac{1}{2} \int\int_{F} |\nabla f |^{2} ~ dA.
\]

In local coordinates the energy can be obtained by integrating the $2$--form
\[
\frac{1}{2}g^{ij}h_{\alpha\beta}\frac{\partial f^{\alpha}}{\partial x^{i}%
}\frac{\partial f^{\beta}}{\partial x^{j}}dx_{i}dx_{j}.
\]
If we take $e_{1},e_{2}$ to be an orthonormal frame in a neighborhood on
$(F,g)$, then this expression simplifies to
\[
\frac{1}{2}(||f_{\ast}(e_{1})||^{2}+||f_{\ast}(e_{2})||^{2})dx_{1}dx_{2}.
\]

We now define a corresponding notion of simplicial energy. Suppose that $f$ is
a map from a metrized or quasi-metrized triangulated compact surface $F$ to a
Riemannian manifold $M$. If an edge $e_{i}$ has $l_{i}>0$ then we denote the
ratio $L_{i}/l_{i}$ by $\sigma_{i}$, and call it the \emph{stretch factor} of
the edge $e_{i}$ under $f$. If $e_{i}$ has length $l_{i}=0$ then we define the
stretch factor by $\sigma_{i}=0$ if $L_{i}=0$, and $\sigma_{i}=\infty$ if
$L_{i}>0$.

By analogy with the definition of smooth energy, we define the formula for the
simplicial energy of a map on a triangle by multiplying the squares of stretch
factors by the local contribution to the area.

\begin{definition}
\label{defnofsimplicialenergy}The \emph{simplicial energy} of a
piecewise-smooth map $f$ of a metrized triangulated compact surface
$(F,\tau,l)$ to a Riemannian manifold $M$ is
\[
E_{S}(f)=\frac{1}{2}\sum_{T\in\tau}(\sigma_{i}^{2}+\sigma_{j}^{2})l_{i}%
l_{j}+(\sigma_{j}^{2}+\sigma_{k}^{2})l_{j}l_{k}+(\sigma_{k}^{2}+\sigma_{i}%
^{2})l_{k}l_{i},
\]
where $l_{i},l_{j},l_{k}$ are the edge lengths of the $2$--simplex $T$ of
$\tau$.
\end{definition}

In the case where some edge $e_{i}$ has length $l_{i}=0$ in a quasi-metric and
this edge is not mapped by $f$ to a point, so that $L_{i}\neq0$ and
$\sigma_{i}=\infty$, we define the simplicial energy of $f$ to be infinite.
For all other cases the above formula gives a well defined finite energy.

\begin{remark}
The simplicial energy of $f$ is invariant under scale change of the domain
simplicial metric. If the simplicial metric $l$ is replaced by $\lambda l$,
resulting in multiplying each $l_{i}$ by a constant $\lambda>0$, then the
energy is unchanged.
\end{remark}

It is sometimes convenient to rewrite the simplicial energy as a sum over the
edges of $\tau$. For a finite energy map $f:(F,\tau,l)\rightarrow M$, this gives%

\begin{equation}
E_{S}(f)=\frac{1}{2}\sum_{e_{i}\in\tau}\left(  \frac{l_{i_{1}}+l_{i_{2}%
}+l_{i_{3}}+l_{i_{4}}}{l_{i}}\right)  ~L_{i}^{2} \label{edgesum}%
\end{equation}

where the sum is over the edges $\{e_{i}\}$ of $\tau$, and $l_{i_{1}}$,
$l_{i_{2}}$, $l_{i_{3}}$ and $l_{i_{4}}$ are the lengths of the four edges of
$\tau$ which are adjacent to $e_{i}$ in the two triangles of which $e_{i}$ is
a face. See Figure~\ref{triangle.fig}.

\begin{figure}[ptbh]
\includegraphics[width=2in]{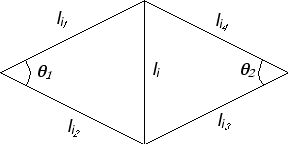}\caption{Edge lengths adjacent to
$e_{i}$ }%
\label{triangle.fig}%
\end{figure}

The coefficient $w_{i}=(l_{i_{1}}+l_{i_{2}}+l_{i_{3}}+l_{i_{4}})/2l_{i}$ of
$L_{i}^{2}$ is called the \emph{weight} of $e_{i}$. Note that for any
simplicial metric on $(F,\tau)$, each weight $w_{i}$ is positive. Other
approaches to combinatorial harmonic maps have studied a variety of ways of
assigning such weights \cite{PinkallPolthier:1993, DaskalopoulosMese,
IzekiNayatani, Wang:98}. For example the combinatorial energy of Pinkall and
Polthier \cite{PinkallPolthier:1993} assigns to the edge $e_{i}$ the weight
$w_{i}=(\cot\theta_{1}+\cot\theta_{2})/4$, where $\theta_{1}$ and $\theta_{2}$
are the two Euclidean angles opposite the edge $e_{i}$. See
Figure~\ref{triangle.fig}. The cotangent formula goes back at least to Duffin
\cite{Duffin}. Note that some of these weights will be negative if any of the
triangles of $\tau$ has an obtuse angle. Our choice of weights was made
because of its simplicity and its close connection to the previously defined
simplicial area. For our purposes, it also turns out to be important that all
our weights are positive. In fact, many of the results in this paper hold for
any energy functional of the form $\sum w_{i}L_{i}^{2}$, so long as all the
weights are positive, though the connection between simplicial energy and area
does not.

For our next result, it will be convenient to introduce the following
notation. Let $F$ be a surface with triangulation $\tau$, and let $l$ be a
simplicial quasi-metric on $(F,\tau)$. Then we denote by $\tau_{0}$ the
subcomplex of $\tau$ which consists of all zero length edges.

\begin{lemma}
\label{E>A} Let $f:F\rightarrow M$ be a piecewise-smooth map of a connected
triangulated compact surface $(F,\tau)$ to a Riemannian manifold $M$. Let $l$
be a simplicial quasi-metric on $(F,\tau)$. Then
\[
E_{S}(f)\geq A_{S}(f).
\]

Further equality holds if and only the stretch factors $\sigma_{i}$ are
constant on the edges of each component of $\tau-\tau_{0}$. In particular, if
$l$ is a simplicial metric, so that no $l_{i}$ is zero, then equality holds if
and only if all stretch factors $\sigma_{i}$ take the same value on all edges
of $\tau$.
\end{lemma}

\begin{remark}
This result is the combinatorial analogue of the fact that if $F\ $and $M$
have Riemannian metrics and $f:F\rightarrow M$ is a smooth map, then the
energy of $f$ is always at least as large as its area, with equality if and
only if $f$ is conformal. The condition that the stretch factors $\sigma_{i}$
take the same value on all edges of $\tau$ is a much weaker condition than
conformality, but this may be an advantage for some purposes.
\end{remark}

\begin{proof}
Suppose that edges $e_{i}$ and $e_{j}$ lie on the boundary of a triangle in
$\tau$. If $l_{i}$ and $l_{j}$ are positive, we apply the inequality
$a^{2}+b^{2}\geq2ab$, with equality precisely when $a=b$. We obtain
\begin{equation}
\frac{1}{2}(\sigma_{i}^{2}+\sigma_{j}^{2})l_{i}l_{j}\geq\sigma_{i}\sigma
_{j}l_{i}l_{j}=\frac{L_{i}}{l_{i}}\frac{L_{j}}{l_{j}}l_{i}l_{j}=L_{i}L_{j},
\label{ab}%
\end{equation}
with equality precisely when $\sigma_{i}=\sigma_{j}$. If some $l_{i}$ is zero,
we have two cases. If $l_{i}=0$, and $L_{i}>0$, then $\sigma_{i}$ is infinite,
so that $\frac{1}{2}(\sigma_{i}^{2}+\sigma_{j}^{2})l_{i}l_{j}\geq L_{i}L_{j}$.
If $\sigma_{j}$ is finite, $l_{i}=0$, and $L_{i}=0$, then $\sigma_{i}=0$ and
the same inequality holds as both sides are zero. Thus this inequality holds
in all cases. Summing over all triangles of $F$ shows that $E_{S}(f)\geq
A_{S}(f)$.

If $E_{S}(f)=A_{S}(f)$, and $\tau_{0}$ is empty, the connectedness of the
$1$--skeleton of $\tau$ immediately implies that all stretch factors
$\sigma_{i}$ take the same value on all edges of $\tau$. If $\tau_{0}$ is
non-empty, the connectedness of the $1$--skeleton of the dual cellulation of
$\tau$ implies that all the stretch factors $\sigma_{i}$ are constant on the
edges of each component of $\tau-\tau_{0}$.
\end{proof}

We now give an example that shows some of the difficulties in achieving useful
simplicial versions of area and energy. The example illustrates that it is not
possible to capture all aspects of how a map stretches a surface using only
information about the stretching of edges in a fixed triangulation.

\begin{example}
\label{Example.stretch} Measuring stretching along the boundary of a triangle
fails to capture Dirichlet energy.
\end{example}

Let $f:T_{1}\rightarrow T_{2}$ be a simplicial map from an isosceles triangle
with edges of length $1,1,2-\epsilon$ to an equilateral triangle with edge
lengths all equal to one. The Euclidean area of $T_{1}$ is close to zero,
while that of $T_{2}$ is $\sqrt{3}/4$. The affine map $f$ that takes $T_{1}$
to $T_{2}$ stretches each of these edges by a factor of $1$ or less. If we use
a combinatorial definition of area that makes the area of $T_{1}$ close to
zero, and a combinatorial definition of the energy of $f$ which depends only
on how much $f$ stretches the edges of $T_{1}$, then the energy of $f$ would
also be close to zero. Thus the energy of $f$ would be much less than its
area, which is the area of the image triangle $T_{2}$, unlike the situation of
Lemma \ref{E>A}. The problem is that measuring the stretching of $f$ only on
the boundary of $T_{1}$ ignores the large stretching by $f$ of segments in
$T_{1}$ orthogonal to the long edge.

We conclude that reasonable formulas defining simplicial area and energy
cannot simultaneously satisfy all of the following:

\begin{enumerate}
\item Simplicial energy is evaluated using stretching of edges and simplicial area.

\item Simplicial energy is greater or equal to simplicial area.

\item The simplicial area of a triangle contained in a line is zero.
\end{enumerate}

As we pointed out in section \ref{simplicialarea} our definition of simplicial
area only has the third property when every edge of a triangle has zero
length. In section \ref{families}, we will see how using quadrilaterals
instead of triangles removes some of these difficulties.

\section{Polygonal Decompositions\label{polygonaldecompositions}}

In this section, we generalise the preceding two sections to the situation
where a compact surface $F$ is divided into polygons rather than just
triangles. We again use $\tau$ to denote such a division. We again have a map
$l$ that assigns to each edge $e_{i}$, $1\leq i\leq r$, of $\tau$ a length
$l_{i}=l(e_{i})>0$. For each polygon of $\tau$, we require that each edge has
length no more than the sum of all the other edges of that polygon. This
suffices to ensure the existence of a convex embedding of the polygon in the
Euclidean plane. If the polygon has more than three edges, we cannot expect
this embedding to be unique, but this is not important. Note that we allow the
possibility that for some polygons one of these inequalities is an equality.
We call such an assignment a \emph{simplicial metric} on $F$, and denote it by
$(F,\tau,l)$, or just $l$ when the context is clear. As in the previous
sections, we do not require that each polygon of $\tau$ be embedded in $F$. We
will also extend this to the case where some edges of $\tau$ are assigned zero
length. We call such an extension a \emph{simplicial quasi-metric} on $F$.

Throughout this section, $F$ will be a compact surface with a polygonal
decomposition $\tau$ and a simplicial metric or quasi-metric $l$.

Let $P\ $be a $n$--gon, whose edges are $e_{1},e_{2},\ldots,e_{n}$ with
lengths $l_{1},l_{2},\ldots,l_{n}$ reading round $P$. Thus $l_{i}$ and
$l_{i+1}$ are lengths of adjacent edges, where we define $l_{n+1}$ to equal
$l_{1}$. We define the \emph{simplicial area} $A_{S}$ of this polygon by the
formula%
\[
A_{S}(P)=\sum_{i=1}^{n}l_{i}l_{i+1}.
\]

As in the case of triangles, the simplicial area of polygons is not, in
general, additive. However for quadrilaterals there is a natural analogue of
the conformal subdivision of a triangle which we described near the start of
section \ref{simplicialarea}. But note that it may be that the process we
describe does not correspond to subdividing a quadrilateral embedded in the
plane. Given a quadrilateral $P$, there is a natural way to divide it into
four sub quadrilaterals, by cutting along two lines which join interior points
of opposite edges of $P$. We now assign lengths to the new edges as shown in
Figure~\ref{quads}.

\begin{figure}[ptbh]
\centering
\includegraphics[width=2in]{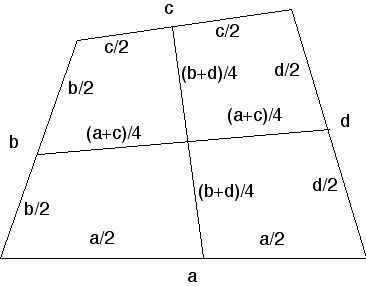} \caption{Conformal subdivision
preserves area.}%
\label{quads}%
\end{figure}

As for triangles, the area of a polygon can be non-zero even if several edges
have length zero. However the following fact will be very useful. If $P$ is a
quadrilateral, and two opposite edges of $P$ are assigned zero length, then
the simplicial area of $P$ is zero. For the applications in sections
\ref{families} and \ref{genus}, this is all we will need.

\begin{definition}
The \emph{simplicial area} of $(F,\tau,l)$ is
\[
A_{S}(F,\tau,l)=\sum_{P\in\tau}A_{S}(P).
\]

\end{definition}

Let $f:F\rightarrow M$ be a piecewise-smooth map to a Riemannian manifold $M$
and let $L_{i}$ denote the length of $f|e_{i}$. Of course, this can only be
defined when $f|e_{i}$ is rectifiable.

\begin{definition}
For a $n$--gon $P$ of $\tau$, the \emph{simplicial area} of a map
$f:P\rightarrow M$ is $\sum_{i=1}^{n}L_{i}L_{i+1}$, which we denote by
$A_{S}(f,P)$.

The \emph{simplicial area} of a map $f:(F,\tau,l)\rightarrow M$ is
\[
A_{S}(f)=\sum_{P\in\tau}A_{S}(f,P).
\]

\end{definition}

Next we define the simplicial energy of $f$. As for triangles, we use the
stretch factor $\sigma_{i}=L_{i}/l_{i}$, when $l_{i}>0$. If $l_{i}=0$ we
define the stretch factor by $\sigma_{i}=0$ if $L_{i}=0$, and $\sigma
_{i}=\infty$ if $L_{i}>0$.

\begin{definition}
For a $n$--gon $P$ of $\tau$, the \emph{simplicial energy} of a map
$f:P\rightarrow M$ is $E_{S}(f,P)=\sum_{i=1}^{n}(\sigma_{i}^{2}+\sigma
_{i+1}^{2})l_{i}l_{i+1}$, where $\sigma_{n+1}$ equals $\sigma_{1}$.

The \emph{simplicial energy} of a piecewise-smooth map $f$ of $(F,\tau,l)$ to
a Riemannian manifold $M$ is
\[
E_{S}(f)=\frac{1}{2}\sum_{P\in\tau}E_{S}(f,P).
\]

\end{definition}

As for triangles, the above formula gives a well defined finite energy, unless
some $\sigma_{i}=\infty$, when $E_{S}(f)=\infty$.

\begin{remark}
The simplicial energy of $f$ is invariant under scale change of the domain
simplicial metric. If the simplicial metric $l$ is replaced by $\lambda l$,
resulting in multiplying each $l_{i}$ by a constant $\lambda>0$, then the
energy is unchanged.
\end{remark}

It is sometimes convenient to rewrite the simplicial energy as a sum over the
edges of $\tau$. For a finite energy map $f:(F,\tau,l)\rightarrow M$, this gives%

\begin{equation}
E_{S}(f)=\frac{1}{2}\sum_{e_{i}\in\tau}\left(  \frac{l_{i_{1}}+l_{i_{2}%
}+l_{i_{3}}+l_{i_{4}}}{l_{i}}\right)  ~L_{i}^{2}%
\end{equation}
where the sum is over the edges $\{e_{i}\}$ of $\tau$, and $l_{i_{1}}$,
$l_{i_{2}}$, $l_{i_{3}}$ and $l_{i_{4}}$ are the lengths of the four edges of
$\tau$ which are adjacent to $e_{i}$ in the two polygons of which $e_{i}$ is a
face. See Figure~\ref{triangle.fig} for the case when the polygons are triangles.

If $l$ is a simplicial quasi-metric on $(F,\tau)$, then we denote by $\tau
_{0}$ the subcomplex of $\tau$ which consists of all zero length edges. The
following result is the natural generalisation of Lemma \ref{E>A}, and is
proved in the same way.

\begin{lemma}
\label{E>Aforpolygons}Let $f:(F,\tau,l)\rightarrow M$ be a piecewise-smooth
map of $F$ to a Riemannian manifold $M$. Then
\[
E_{S}(f)\geq A_{S}(f).
\]

Further equality holds if and only the stretch factors $\sigma_{i}$ are
constant on the edges of each component of $\tau-\tau_{0}$. In particular, if
$l$ is a simplicial metric, so that no $l_{i}$ is zero, then equality holds if
and only if all stretch factors $\sigma_{i}$ take the same value on all edges
of $\tau$.
\end{lemma}

\section{Simplicial harmonic maps}

\label{simplicialharmonic}

We call a map of a surface $F$ into a manifold $M$ \emph{trivial} if the image
of $\pi_{1}(F)$ is contained in a cyclic subgroup of $\pi_{1}(M)$. When $M$ is
non-positively curved, or more generally when $\pi_{2}(M)=0$, this is
equivalent to the map factoring through a map of the circle into $M$.

In this section we show that any homotopy class of maps from a metrized
triangulated closed surface $(F,\tau,l)$ into a closed non-positively curved
Riemannian manifold $M$ contains a simplicial-energy minimizing map. This map
is unique if the map is nontrivial and the curvature of $M$ is strictly
negative. Note that nontrivial maps never exist if $F$ is the sphere or
projective plane. Further if the metric on $M$ is strictly negatively curved,
then nontrivial maps do not exist if $F$ is the torus or Klein bottle.

Since $M$ is geodesically complete, any map $f_{0}:F\rightarrow M$ can be
homotoped, fixing the vertices of $F$, to a new map that sends each edge of
$\tau$ to a geodesic arc in $M$. If $M$ is non-positively curved, such an arc
is unique. After picking an identification of each edge $e_{i}$ of $\tau$ with
the interval $[0,l_{i}]$ of the real line, this yields a canonical map from
the $1$--skeleton of $\tau$ to $M$. The resulting map $f_{1}$ has energy no
greater than the original map $f_{0}$, and it suffices to consider such maps
in searching for an energy minimizer. The choice of an extension of $f_{1}$ to
the $2$--simplices of $F$ does not affect its energy, so we are free to make
this choice in any convenient way. To uniquely describe an extension, we make
some further choices. Since the edge lengths of $\tau$ satisfy the triangle
inequality, we can identify each triangle of $\tau$ with the Euclidean
triangle having the same edge lengths. As discussed at the start of
Section~\ref{simplicialarea}, each triangle of $\tau$ has a preferred vertex,
and we fix the natural cone structure on the triangle with this vertex as cone
point. There is a unique extension of $f_{1}$ to a map of $F$ that sends each
cone line in a triangle to a constant speed geodesic arc in $M$.

If $\tau$ is a polygonal decomposition of $F$, we proceed in the same way. The
only difference is that we need to choose a preferred vertex for each polygon,
choose a convex embedding of each polygon in the Euclidean plane, and then use
the natural cone structure on the polygon with this vertex as cone point.

We make the following definition.

\begin{definition}
\label{defnofsimplicialmap}A map $f:(F,\tau,l)\rightarrow M$ from a
triangulated surface $F$ with a simplicial metric $l$ to a Riemannian manifold
$M$ is \emph{simplicial} if each edge of $\tau$ and each cone line in each
$2$--simplex of $\tau$ is mapped as a geodesic arc.
\end{definition}

We can use the same definition and terminology if $\tau$ is a polygonal
decomposition of $F$.

The preceding discussion shows that if $M$ is non-positively curved, then any
map $f_{0}:F\rightarrow M$ can be homotoped, leaving the vertices fixed, to a
unique simplicial map $f_{1}$ and that the simplicial energy of $f_{1}$ is no
greater than that of $f_{0}$. Thus to establish the existence of
simplicial-energy minimizing maps into a non-positively curved manifold, we
need only consider simplicial maps. Given a simplicial map $f$ we define a
\emph{variation} of $f$ to be a $1$--parameter family of maps $f_{t}%
,~-\epsilon<t<\epsilon$ with $f_{0}=f$. Note that $f_{t}$ need not be
simplicial when $t\neq0$.

\begin{definition}
A map $f:(F,\tau,l)\rightarrow M$ of a closed surface with a simplicial metric
or simplicial quasi-metric to a Riemannian manifold $M$ is \emph{simplicial
harmonic} if it is simplicial and is a critical point of simplicial energy
under all variations $f_{t}$ of $f$.

If $F$ is compact with non-empty boundary, then $f$ is \emph{simplicial
harmonic} if it is simplicial and is a critical point of simplicial energy
under all variations $f_{t}$ of $f$ which leave $\partial F$ fixed.

The map $f$ is \emph{simplicial conformal} if the stretch factors $\sigma_{i}$
are constant on the edges of each component of $\tau-\tau_{0}$, where
$\tau_{0}$ denotes the subcomplex of $\tau$ which consists of all zero length edges.
\end{definition}

\begin{remark}
Lemmas \ref{E>A} and \ref{E>Aforpolygons} tell us that if $f$ is a simplicial
map, then $E_{S}(f)=A_{S}(f)$ if and only if $f$ is simplicial conformal.
\end{remark}

\begin{remark}
The above definition of $f$ being simplicial harmonic is equivalent to the one
obtained by restricting attention to variations in which each map $f_{t}$ is simplicial.
\end{remark}

For most of this paper, we will consider only the case when the surface $F$ is
closed, but in the next section it will be important to consider the case when
$F$ has non-empty boundary.

\begin{proposition}
\label{uniqueharmonicmapsexist} Let $(F,\tau,l)$ be a closed triangulated
surface with a simplicial metric. Let $f:(F,\tau,l)\rightarrow M$ be a
continuous map of $F$ to a non-positively curved, closed Riemannian manifold
$M$. Then

\begin{enumerate}
\item $f$ is homotopic to a simplicial map $f_{0}$ that minimizes simplicial
energy among all maps homotopic to $f$, and hence is simplicial harmonic.

\item If $f$ is nontrivial and $M$ is negatively curved, then $f_{0}$ is the
unique simplicial harmonic map in the homotopy class of $f$.
\end{enumerate}

Furthermore, if $l$ is a simplicial quasi-metric and $f$ is homotopic to a map
with finite energy, then (1) and (2) continue to hold.
\end{proposition}

\begin{proof}
(1) Any homotopy of a map $f:F\rightarrow M$ that decreases some image length
$L_{i}$ without increasing any $L_{j}$ reduces the simplicial energy (see
Equation~\ref{edgesum}). It follows that a map that minimizes simplicial
energy sends each edge of $\tau$ to an arc in $M$ that minimizes length in its
relative homotopy class.

A simplicial map is not determined by the images of the vertices of $\tau$,
since distinct relative homotopy classes of arcs connecting two points will
yield distinct geodesic arcs connecting those points. However a simplicial map
$f$ is determined uniquely by the images of the vertices and the homotopy
class (rel boundary) of $f$ on each edge of $\tau$.

We select a vertex $v\in F$ as a basepoint and also fix a fundamental region
$W$ in the universal cover $\widetilde{M}$ of $M$. Since $M$ is compact, so is
$W$. Let ${\mathcal{F}}$ be the set of simplicial maps homotopic to $f$, let
$E_{0}>0$ be a constant and let ${\mathcal{F}}_{0}$ be the subset of
${\mathcal{F}}$ consisting of maps with simplicial energy less than or equal
to $E_{0}$. We will show that ${\mathcal{F}}_{0}$ is compact.

Let $f$ be a map in ${\mathcal{F}}_{0}$, let $P$ denote $f(v)$ and let
$\widetilde{P}$ be a point of $W$ projecting to $P$. Let $\widetilde{F}$
denote the universal cover of $F$, equipped with the natural triangulation
lifted from $\tau$, and let $\widetilde{v}$ denote a point of $\widetilde{F}$
above $v$. There is a unique map $\widetilde{f}:\widetilde{F}\rightarrow
\widetilde{M}$ that covers $f$ and sends $\widetilde{v}$ to $\widetilde{P}$.
As $\widetilde{M}$ is simply connected, the simplicial map $\widetilde{f}$ is
determined completely by the images of the vertices of $\widetilde{F}$.

Each edge $e_{i}$ in $\tau$ is part of a simple arc $\gamma_{i}$ contained in
the $1$--skeleton of $\tau$ and connecting $e_{i}$ to the base point $v$. The
restriction of $f$ to $\gamma_{i}$ lifts to a path $\widetilde{\gamma_{i}}$ in
$\widetilde{M}$ which starts at $\widetilde{P}$.

Let $l_{m}$ be the length of the shortest edge in the given simplicial metric
on $\tau$, let $l_{M}$ be the length of the longest edge and let $L_{i}$ be
the length of the restriction of $f$ to $e_{i}$. As $f$ lies in ${\mathcal{F}%
}_{0}$, we have
\[
(L_{i}/l_{i})^{2}l_{m}^{2}\leq(L_{i}/l_{i})^{2}l_{i}l_{j}=\sigma_{i}^{2}%
l_{i}l_{j}\leq2E_{0},
\]
so that
\[
L_{i}\leq\sqrt{2E_{0}}(l_{i}/l_{m})\leq\sqrt{2E_{0}}(l_{M}/l_{m}).
\]
Let $L_{M}$ denote the quantity $\sqrt{2E_{0}}(l_{M}/l_{m})$, so that
$L_{i}\leq L_{M}$, for all $i$. Since $\gamma_{i}$ is simple it contains at
most $r$ edges, where $r$ is the total number of edges in $\tau$. Hence the
restriction of $f$ to $\gamma_{i}$ has length at most $rL_{M}$. Note that this
bound applies for any element $f$ of ${\mathcal{F}}_{0}$. As the point
$\tilde{P}$ lies in the compact fundamental domain $W$, there is a ball $B$ in
$\widetilde{M}$ such that each $\widetilde{\gamma_{i}}$ lies in $B$, for any
$f\in{\mathcal{F}}_{0}$. It follows that the simplicial map $\widetilde{f}$ is
completely determined by the finite number of points in $B$ that are vertices
of the $\widetilde{\gamma_{i}}$'s. The set of all such is compact, therefore a
minimum energy map exists, which establishes part (1).

Note that the above argument does not really use the non-positive curvature
assumption on $M$. Without that assumption, one can still find an energy
minimizing map homotopic to $f$, but may not be able to arrange that this map
is simplicial as defined in Definition \ref{defnofsimplicialmap}.

(2) Now suppose that $f$ is nontrivial, and that $f_{0}$ and $f_{1}$ are
simplicial harmonic maps in the homotopy class ${\mathcal{F}}$ of $f$. Further
suppose that $M$ is negatively curved. As $f_{0}$ and $f_{1}$ are homotopic,
there is a \textquotedblleft straight\textquotedblright\ homotopy $f_{t}$
between them, in which $f_{t}$ is a simplicial map for each $0\leq t\leq1$ and
the images of vertices of $\tau$ move along geodesic arcs in $M$ at constant speed.

We will make use of the following key fact. Let $\alpha(t)$ and $\beta(t)$
denote two geodesic arcs in the negatively curved manifold $M$ and let
$\gamma(t)$ be the shortest geodesic arc from $\alpha(t)$ to $\beta(t)$. Then
the length $L(t)$ of $\gamma(t)$ is a convex function of $t$. Further it is
strictly convex except possibly when the arcs $\alpha(t)$, $\beta(t)$ and
$\gamma(t)$ are all contained in a single geodesic \cite{Thurston}. Since each
$L_{i}(t)$ is convex we see that the simplicial energy function
\[
E_{S}(f_{t})=\sum w_{i}{L_{i}(t)}^{2},~~~w_{i}>0,
\]
given in Equation~\ref{edgesum} is also a convex function of $t$.

Now suppose that $f_{0}$ and $f_{1}$ do not coincide. Then there is a vertex
$v$ of $\tau$ such that the homotopy $f_{t}$ from $f_{0}$ to $f_{1}$ moves $v$
along a nontrivial geodesic arc $\mu$. We let $\lambda$ denote the unique
complete geodesic of $M$ which contains $\mu$. As $f$ is nontrivial, there is
an edge $e_{i}$ of $\tau$ whose image under $f_{0}$ or $f_{1}$ is not
contained in $\lambda$. First we suppose that $v$ is a vertex of $e_{i}$. In
this case, the function $L_{i}(t)$ is a strictly convex function of $t$, and
it follows immediately that $E_{S}(f_{t})$ is also strictly convex, and so can
have only one critical point. Since both $f_{0}$ and $f_{1}$ are critical
points, they must coincide, proving part (2) of the proposition in this case.
However it may be the case that every edge of $\tau$ incident to $v$ has image
which remains in $\lambda$ under each map $f_{t}$. Here we need a slightly
refined version of the above key fact. Let $\alpha(t)$ and $\beta(t)$ denote
two geodesic arcs in the negatively curved manifold $M$ which are contained in
a single geodesic $\lambda$, and suppose, in addition, that $\alpha$ and
$\beta$ each has constant speed. Let $\gamma(t)$ be the shortest geodesic arc
from $\alpha(t)$ to $\beta(t)$. Then the length $L(t)$ of $\gamma(t)$ is a
linear function of $t$. Hence $L(t)^{2}$ is a strictly convex function of $t$,
unless $L(t)$ is constant.

Now we return to the situation where $v$ is a vertex of $\tau$ such that the
homotopy $f_{t}$ moves $v$ along a nontrivial geodesic arc contained in
$\lambda$, there is an edge $e_{i}$ of $\tau$ whose image under $f_{0}$ or
$f_{1}$ is not contained in $\lambda$, and $v$ is not a vertex of $e_{i}$.
Choose a simple path in the $1$-skeleton of $\tau$ that connects $v$ to
$e_{i}$, let $e_{j}$ denote the first edge of this path and let $w$ denote the
other vertex of $e_{j}$. The preceding discussion shows that $L_{j}(t)^{2}$ is
a strictly convex function of $t$ unless $f_{t}(w)$ stays in $\lambda$ for all
$t$, and $L_{j}(t)$ is constant. In the first case, it follows as before that
$E_{S}(f_{t})$ is a strictly convex function of $t$. In the second case, it
follows that $w$, like $v$, is a vertex of $\tau$ such that the homotopy
$f_{t}$ from $f_{0}$ to $f_{1}$ moves $w$ along a geodesic arc contained in
$\lambda$. Further this geodesic arc is nontrivial because $L_{j}(t)$ is
constant. As $w$ is joined to $e_{i}$ by a shorter edge path than $v$, a
simple inductive argument completes the proof that $E_{S}(f_{t})$ is a
strictly convex function of $t$ in all cases. As above this implies that
$f_{0}$ and $f_{1}$ coincide, which completes the proof of part (2) of the proposition.

Finally, we consider the case when $l$ is a quasi-metric. It may no longer be
true that a harmonic map homotopic to $f$ exists. For example, $l$ might
assign zero length to each edge of $\tau$, which would mean that all maps
homotopic to $f$ have infinite energy unless $f$ is homotopic to a point.
However if $f$ is homotopic to a map with finite energy, then we can take
$l_{m}$ to be the smallest positive length of the edges in $\tau$. All zero
length edges must be mapped to a point if the total energy is finite, and so
they make no contribution to the energy. The previous argument then applies
and establishes that parts (1)\ and (2) of the theorem continue to hold in
this case.
\end{proof}

\section{Mean value and convex hull properties}

\label{meanvalueandconvexhullproperties}

In this section, we show that simplicial harmonic maps share some of the
classical properties of smooth harmonic maps.

Let $F$ be a Riemannian manifold. A map $f:F\rightarrow{\mathbb{R}}^{n}$ has
the \textit{mean value property} if for any point $x\in F$, and any ball $B$
of radius $r$ centered at $x$, the value of $f$ at $x$ is the average of the
values of $f$ over $\partial B$. If $f$ is harmonic it has the mean value
property. Harmonic maps to ${\mathbb{R}}^{n}$ also have the following
\textit{convex hull property}: for $D$ any compact connected submanifold of
$F$, the image $f(D)$ is contained in the convex hull of $f(\partial D)$. Note
that if $F$ is closed, then any harmonic map from $F$ to ${\mathbb{R}}^{n}$
must be constant, so the discussion in this section is only of interest when
$F\ $has non-empty boundary.

Finally a real valued function $f$ on $F$ is said to satisfy the \emph{maximum
principle} if, for any compact submanifold $D$ of $F$, the restriction of $f$
to $D$ attains its upper and lower bounds only on $\partial D$. Non-constant
harmonic functions also satisfy the maximum principle.

As mentioned in the introduction, the discrete harmonic maps of Pinkall and
Polthier \cite{PinkallPolthier:1993} do not satisfy the convex hull or mean
value properties. See \cite{PolthierRossman:2000} for an interactive
demonstration of an example found by Polthier and Rossman
\cite{PolthierRossman:2002} which shows that discrete harmonic maps do not
satisfy the convex hull property.

Next we show that simplicial harmonic maps do satisfy the convex hull and mean
value properties. Let $v$ be a vertex of a surface triangulation $\tau$, and
let $D$ denote the closed star of $v$, so that $D$ is the union of all the
simplices of $\tau$ which contain $v$. We let $v_{1},\ldots,v_{k}$ denote the
other vertices of $D$.

\begin{lemma}
\label{convexhullproperty}(Euclidean Mean Value and Convex Hull Properties)
Let $f$ be a simplicial harmonic map of a connected metrized triangulated
compact surface $(F,\tau,l)$ into ${\mathbb{R}}^{n}$. Then $f(v)$ lies in the
convex hull of $f(v_{1}),\ldots,f(v_{k})$. Further $f(v)$ is the centre of
mass of a collection of positive weights placed at $f(v_{1}),\ldots,f(v_{k})$.
\end{lemma}

\begin{proof}
Equation (1) shows that the simplicial energy $E_{S}(f)$ has the form $\sum
w_{i}{L_{i}}^{2}$, for some positive constants $w_{i}$. As $f$ is harmonic, it
is a critical point of the functional $E_{S}$. We will apply this fact to
variations of $f$ through simplicial maps which equal $f$ on all vertices
apart from $v$. For $1\leq j\leq k$, let $e_{j}$ denote the edge joining $v$
to $v_{j}$, and consider the terms $\sum_{j=1}^{k}w_{j}{L_{j}(t)}^{2}$ of
$E_{S}(f_{t})$. Let $(x_{1},\ldots,x_{n})$ be the coordinates of $f(v)$, and
let $(x_{1j},\ldots,x_{nj})$ be the coordinates of $f(v_{j})$, $1\leq j\leq
k$. Thus $\sum_{j=1}^{k}w_{j}{L_{j}(t)}^{2}=\sum_{j=1}^{k}w_{j}\sum_{i=1}%
^{n}(x_{i}-x_{ij})^{2}$. As $f$ is a critical point of $E_{S}$, we know that,
for $1\leq i\leq n$, the partial derivative of this summation with respect to
$x_{i}$ must be zero. Hence, for $1\leq i\leq n$, we have $\sum_{j=1}^{k}%
w_{j}2(x_{i}-x_{ij})=0$, so that $x_{i}\sum_{j=1}^{k}w_{j}=\sum_{j=1}^{k}%
w_{j}x_{ij}$. As each $w_{j}$ is positive, $x_{i}$ is a linear combination of
the $x_{ij}$'s with coefficients which lie between $0$ and $1$. It follows
that $f(v)$ lies in the convex hull of $f(v_{1}),\ldots,f(v_{k})$, and that
$f(v)$ is the centre of mass of weights $w_{1},\ldots,w_{n}$ placed at
$f(v_{1}),\ldots,f(v_{k})$.
\end{proof}

The special case of this result when $n=1$ gives the following form of the
Maximum Principle.

\begin{lemma}
(Maximum Principle) Let $f$ be a non-constant simplicial harmonic map of a
connected metrized triangulated compact surface $(F,\tau,l)$ into
${\mathbb{R}}$. Then for any subsurface $E$ of $F$ which is also a subcomplex,
the restriction of $f$ to $E$ attains its upper and lower bounds only on
$\partial E$.
\end{lemma}

The convex hull of a subset of a general Riemannian manifold $M$ may be rather
complicated. For simplicity we consider only the case of non-positive
curvature and will restrict attention to convex balls in $M$. In this setting
we obtain the following result.

\begin{lemma}
\label{convexhull} Let $f$ be a simplicial harmonic map of a connected
metrized triangulated compact surface $(F,\tau,l)$ into a convex ball $B$ with
a non-positively curved Riemannian metric. Let $v$ be an interior vertex of
$\tau$ and let $D$ denote the closed star of $v$. Let $v_{1},\ldots,v_{k}$
denote the other vertices of $D$. Then $f(v)$ lies in the convex hull of
$f(v_{1}),\ldots,f(v_{k})$.
\end{lemma}

\begin{proof}
If $f(v)$ lies outside the convex hull of $f(v_{1}),\ldots,f(v_{k})$, then
there is a shortest geodesic arc $\alpha$ from $f(v)$ to the convex hull. Take
the tangent plane $P$ at $f(v)$ perpendicular to $\alpha$. The geodesics
joining $f(v)$ to $f(v_{j})$ have tangent vectors at $f(v)$ that all lie on
the same side of $P$. Moving $f(v)$ away from $P$ towards this side will
reduce each $L_{j}$ and so reduce the energy of $f$, contradicting the
assumption that $f$ is simplicial harmonic.
\end{proof}

Lemma~\ref{convexhull} implies a universal area bound for homotopy classes of
maps of a fixed surface into a negatively curved manifold. We first recall the
well-known corresponding bound for smooth maps.

\begin{lemma}
\label{arealessthan4pi(g-1)}Let $M$ be a hyperbolic manifold, let $F$ be the
closed orientable surface of genus $g>1$, and let $f:F\rightarrow M$ be any
map. Then $f$ can be homotoped to a map $f_{0}$ with $\mbox{Area}(f_{0}%
)\leq2\pi(2g-2)$.

If $M$ has sectional curvature bounded above by $-a<0$ then we have
$\mbox{Area}(f_{0})\leq2\pi(2g-2)/a$.
\end{lemma}

\begin{proof}
Suppose that $M$ has sectional curvature bounded above by $-a<0$. If $f$ sends
a simple loop on $f$ to a trivial loop in $M$, we compress $f$. Otherwise $f$
is homotopic to a least area immersion in its homotopy class, using the
existence theorems for minimal surfaces due to Schoen and Yau
\cite{SchoenYau:78}, and Sacks and Uhlenbeck \cite{SacksUhlenbeck}. Note that
as $\pi_{2}(M)=0$, the least area maps obtained in \cite{SchoenYau:78} and
\cite{SacksUhlenbeck} must be homotopic to $f$.

The normal curvature $k_{n}$ of a minimal immersion is non-positive. Its
sectional curvature $K\leq-a+k_{n}\leq-a$. By the Gauss-Bonnet Theorem, with
$F_{0}=f_{0}(F)$,
\[
2\pi(2-2g)=2\pi\chi(F_{0})=\int_{F_{0}}K~dA\leq\int_{F_{0}}%
(-a)~dA=-a\mbox{Area}(F_{0})
\]
So $\mbox{Area}(F_{0})\leq2\pi(2g-2)/a$.

The homotopy class of the original uncompressed surface map can be recovered
by adding tubes of zero area.
\end{proof}

We reprove this result using simplicial harmonic maps. In doing so we no
longer need to use the existence results for minimal surfaces. Since our
results apply equally well to non-orientable surfaces, we use the Euler number
rather than the genus of a surface.

\begin{lemma}
\label{arealessthan2pichi}Let $M$ be a hyperbolic manifold, let $F$ be a
closed surface with Euler characteristic $\chi<0$, and let $f:F\rightarrow M$
be any map. Then $f$ can be homotoped to a map $f_{0}$ with $\mbox{Area}(f_{0}%
)\leq2\pi\left\vert \chi\right\vert $.

If $M$ has sectional curvature bounded above by $-a<0$ then we have
$\mbox{Area}(f_{0})\leq2\pi\left\vert \chi\right\vert /a$.
\end{lemma}

\begin{proof}
Suppose that $M$ has sectional curvature bounded above by $-a<0$. Choose a
triangulation and a simplicial metric on $F$, and let $f_{0}$ be a simplicial
harmonic map in the homotopy class of $f$. If $f$ is trivial, then $f_{0}$ has
image contained in some geodesic of $M$ and so has zero Riemannian area. In
this case the required inequality is obvious. Otherwise $f$ is nontrivial, so
that $f_{0}$ is unique and at least one $2$--simplex of $\tau$ does not have
image contained in a geodesic. The simplicial map $f_{0}$ is a ruled smooth
immersion on each such $2$--simplex. Thus the Riemannian metric induced on $F$
has curvature $\leq-a$ on the interior of such $2$--simplices, and edges are
sent to geodesics. At the vertices, the convex hull property in Lemma
\ref{convexhull} shows that the angle sum around each vertex is greater than
$2\pi$, and hence that the vertices contribute negative curvature to the total
curvature when using the Gauss-Bonnet formula for simplicial complexes. The
argument now follows as in the smooth case.
\end{proof}

\begin{remark}
The preceding lemma can be proved in essentially the same way if $F$ is
divided into polygons. For then each polygon has a cone structure, and so has
a division into subcones which are triangles. Again each triangle either has
image contained in a geodesic, or is mapped by a ruled smooth immersion into
$M$.
\end{remark}

\section{Families of maps}

\label{families}

Given an immersion of a surface into a Riemannian manifold, the mean curvature
flow gives an area decreasing flow that pushes the surface in the direction
that decreases its area as rapidly as possible. However singularities develop
as the surface evolves under mean curvature, making it difficult to extend the
flow beyond the time of singularity formation. Additional difficulties arise
in extending mean curvature flow to families of surfaces, and these have
proved an obstruction to topological applications. Harmonic maps offer an
alternative approach to constructing area decreasing flows, a fact exploited
in \cite{HassThompsonThurston}.

In this section we show how to deform a multi-parameter family of simplicial
maps to a family of simplicial harmonic maps. Under this deformation the
simplicial energy of each map in the family decreases monotonically. The
simplicial area may not be monotonically decreasing, but the simplicial area
of each surface is at all times bounded above by its initial value. Although
some of the results in this section work for triangulations, we will use only
decompositions into quadrilaterals. This is because of the applications we
have in mind in this and the next section.

For these applications, it is important that the Riemannian area of a
simplicial map is bounded by its simplicial energy. Note that no such bound
can exist for general maps, as the simplicial energy of a map is determined
solely by the restriction of that map to the $1$--skeleton of the given triangulation.

\begin{lemma}
\label{E>standardA} Let $F$ be a compact surface with a decomposition $\tau$
into quadrilaterals and a simplicial quasi-metric $l$. Let $f:F\rightarrow M$
be a simplicial map to a non-positively curved Riemannian manifold $M$. Then
the Riemannian area, $\mbox{Area}(f)$, satisfies the inequality%
\[
\mbox{Area}(f)\leq E_{S}(f)/2.
\]

\end{lemma}

\begin{proof}
Let $P$ be a quadrilateral of $\tau$, with edges $e_{1},\ldots,e_{4}$
cyclically ordered. Recall that there is a preferred vertex $v$ of $P$, that
we use a cone structure on $P$ with $v$ as cone point, and that a simplicial
map from $F$ to $M$ sends each cone line in this cone structure to a geodesic
in $M$. The diagonal of $P$ that has one end at $v$ divides $P$ into two
triangles, and the restriction of $f$ to each triangle is a ruled $2$--simplex
in $M$. Suppose that $v$ is the vertex $e_{1}\cap e_{4}$, so that one triangle
contains $e_{1}$ and $e_{2}$, and the other triangle contains $e_{3}$ and
$e_{4}$. As $M$ is non-positively curved, the induced curvature on each
triangle is also non-positive and so the induced Riemannian area is no greater
than the Euclidean area of a triangle with the same edge lengths. The
Euclidean areas of the two triangles are at most $L_{1}L_{2}/2$ and
$L_{3}L_{4}/2$. We conclude that the Riemannian area of the restriction of $f$
to $P$ is at most $(L_{1}L_{2}+L_{3}L_{4})/2$, which is clearly at most one
half of the simplicial area of the restriction of $f$ to $P$. Now Lemma
\ref{E>Aforpolygons} implies that this is at most one half of the simplicial
energy of the restriction of $f$ to $P$. Summing over all the quadrilaterals
of $\tau$ immediately yields the inequality $\mbox{Area}(f)\leq E_{S}(f)/2$,
as required.
\end{proof}

Let $F$ be a closed surface with a decomposition $\tau$ into quadrilaterals
and a simplicial quasi-metric $l$. Let $f:(F,\tau,l)\rightarrow M$ be a
continuous map of $F$ to a negatively curved, closed Riemannian manifold $M$.
If $f$ is nontrivial, then Proposition \ref{uniqueharmonicmapsexist} states
that $f$ is homotopic to a unique simplicial harmonic map $f_{0}$.

When $f$ itself is a simplicial map, there is a canonical way to deform it to
the simplicial energy minimizing map $f_{0}$. The simplicial energy function
is determined locally by the images under $f$ of the vertices of $\tau$, and
thus is a real valued function on a finite dimensional manifold. The gradient
of this function determines a direction of deformation for the vertices of
$f(F)$. This gives rise to an energy decreasing flow $f_{t}$ whose derivative
at each $t$ equals the negative of the gradient of the simplicial energy
function. This flow moves the vertex images of $f_{t}(F)$ so as to decrease
simplicial energy as quickly as possible. We define the \emph{simplicial
energy gradient flow} to be this flow of the map $f$. Since there is a unique
critical point for simplicial energy in any homotopy class, the gradient flow
converges to $f_{0}$, the minimizing simplicial harmonic map homotopic to $f$.

We now show that the map $f_{0}$ obtained by minimizing energy among
simplicial maps homotopic to an initial map $f$ depends continuously on the
map $f$ and the simplicial metric $l$. This is a straightforward consequence
of the uniqueness of simplicial harmonic maps within a homotopy class. It was
shown in \cite{EellsLemaire} that smooth harmonic maps depend continuously on
the domain metric.

We use the term \emph{continuous family of maps $f_{s}:F\rightarrow M$ of a
surface $F$ to a manifold $M$ parametrized by a space }$\mathcal{S}$ to refer
to a continuous function $H:\mathcal{S}\times F\rightarrow M$, with
$f_{s}(x)=H(s,x)$, for each $s\in\mathcal{S}$ and $x\in F$.

\begin{proposition}
\label{gradientflow}Let $F$ be a closed surface with a decomposition $\tau$
into quadrilaterals, and let $M$ be a negatively curved closed Riemannian
manifold. Let $f_{s}:(F,\tau,l_{s})\rightarrow M$ be a continuous family of
nontrivial simplicial maps parametrized by a subspace $\mathcal{S}$ of
$\mathbb{R}^{n}$, where $l_{s}$ is a continuous family of simplicial
quasi-metrics. Then the gradient flow applied to each $f_{s}$ yields a
continuous family of simplicial maps $f_{s,t}:(F,\tau,l_{s})\rightarrow M$
parametrized by $\mathcal{S}\times\lbrack0,\infty)$, such that, for each $s$,
the map $f_{s,0}$ equals $f_{s}$, and the family of maps $f_{s,t}$ converges
to a simplicial harmonic map $f_{s,\infty}$. The simplicial harmonic maps
$f_{s,\infty}$ depend continuously on the initial map $f_{s}$ and on the
initial simplicial metric $l_{s}$.
\end{proposition}

\begin{proof}
For fixed $l_{s}$, the simplicial harmonic map homotopic to $f_{s}%
:(F,\tau,l_{s})\rightarrow M$ is unique, so doesn't vary when $f_{s}$ is
changed by a homotopy. The path of maps $f_{s,t}:(F,\tau,l_{s})\rightarrow M$
defined by the gradient flow, for each fixed $s$, varies continuously with the
map, since each point is following the trajectory of a solution to a first
order differential equation whose initial conditions vary continuously.

It remains to show that the simplicial harmonic maps $f_{s,\infty}$ depend
continuously on the choice of simplicial metric $l_{s}$. Let $\{s_{n}\}$ be a
sequence of points in $\mathcal{S}$ converging to $s$ in $\mathcal{S}$, let
$l_{n}$ denote the metric associated to $s_{n}$, and let $l_{\infty}$ denote
the metric associated to $s$. Thus the sequence of simplicial quasi-metrics
$l_{n}$ converges to $l_{\infty}$. Arguing as in the proof of the first part
of Proposition \ref{uniqueharmonicmapsexist} shows that the corresponding
simplicial harmonic maps $f_{n}$ converge to a limiting map $f_{\infty}$ whose
simplicial energy is equal to the limiting value of the simplicial energies of
$f_{n}$. With respect to the simplicial metric $l_{\infty}$, the map
$f_{\infty}$ has the same simplicial energy as $f_{s,\infty}$, which is the
unique harmonic map in its homotopy class with simplicial metric $l_{\infty}$.
Therefore $f_{\infty}$ must be equal to $f_{s,\infty}$. This establishes that
the resulting simplicial harmonic maps vary continuously with the simplicial
metric, as claimed.
\end{proof}

Now we can apply the above result to show that we can homotope a family of
maps from $F$ to $M$ to arrange that there is a uniform bound on the
Riemannian area of the maps in the new family, where the bound depends only on
the Euler characteristic of $F$, and the curvature of $M$.

\begin{theorem}
\label{monotone} Let $F$ be a closed surface with Euler characteristic
$\chi<0$, and with a subdivision $\tau$ into quadrilaterals. Let
$f_{s}:F\rightarrow M$, $s\in{\mathcal{S}}$, be a continuous family of
simplicial maps of $F$ to a Riemannian manifold $M$ with sectional curvatures
bounded above by $-a<0$, where ${\mathcal{S}}$ is a path connected subset of
${\mathbb{R}}^{n}$. Further suppose that some, and hence every, $f_{s}$ is
nontrivial. Then, for each $s\in{\mathcal{S}}$, there is a simplicial
quasi-metric $l^{s}$ on $F$ which varies continuously with $s$, and a
continuous family of maps $g_{s,t}:(F,\tau,l^{s})\rightarrow M$ parametrized
by $\mathcal{S}\times I$, satisfying

\begin{enumerate}
\item $g_{s,0}=f_{s}$,

\item $g_{s,1}$ is a simplicial harmonic map,

\item the simplicial energy of $g_{s,t}$ is non-increasing with $t$,

\item $\mbox{Area}(g_{s,1})\leq2\pi\left\vert \chi\right\vert /a$, for every
$s$.
\end{enumerate}
\end{theorem}

\begin{proof}
For each simplicial map $f_{s}$ in the given family, we pick a simplicial
quasi-metric $l^{s}$ on $\tau$ by setting the edge length $l_{i}^{s}$ of
$e_{i}$ equal to the length $L_{i}^{s}$ of $f_{s}|e_{i}$. For this
quasi-metric, all the stretch factors $\sigma_{i}$ are equal to $1$ on all
edges with $l_{i}^{s}>0$. If $l_{i}^{s}=0$, then $f_{s}$ must map $e_{i}$ to a
point. Thus $f_{s}$ has finite simplicial energy. Now we apply Proposition
\ref{gradientflow}. This yields a continuous family of simplicial maps
$f_{s,t}:(F,\tau,l^{s})\rightarrow M$ parametrized by $\mathcal{S}%
\times\lbrack0,\infty)$, such that $f_{s,0}=f_{s}$. Further, for each $s$, the
function $E_{S}(f_{s,t})$ is a non-increasing function of $t$, and the family
of maps $f_{s,t}$ converges to a simplicial harmonic map $f_{s,\infty}$. Now
we use a homeomorphism between $[0,\infty)$ and $[0,1)$ to obtain a new family
of simplicial maps $g_{s,t}:(F,\tau,l^{s})\rightarrow M$ parametrized by
$\mathcal{S}\times I$, such that $g_{s,0}=f_{s,0}=f_{s}$, and $g_{s,1}%
=f_{s,\infty}$ is simplicial harmonic. This family is continuous by
Proposition \ref{gradientflow}. In particular, parts 1)-3) of the theorem hold.

Since $g_{s,1}$ is simplicial harmonic, the proof of Lemma
\ref{arealessthan2pichi} shows that $\mbox{Area}(g_{s,1})\leq2\pi\left\vert
\chi\right\vert /a$ for every $s$, so that part 4) of the theorem holds.
\end{proof}

The above theorem and its proof are very similar to what occurs in the smooth
setting. The crucial difference is that in the smooth setting, if
$f:F\rightarrow M$ is a map to a Riemannian manifold $M$, one can only define
the Riemannian metric on $F$ induced by $f$ when $f$ is an immersion. In
particular the energy of $f$ must be non-zero. But in the above theorem maps
with zero energy cause no problems, as we simply induce a quasi-metric on $F$.
This greatly simplifies applications, as we will see later in this section and
in the next section.

If we start with a family $f_{s}$ of smooth or piecewise-smooth maps from $F$
to $M$, we need to approximate it by a homotopic family of simplicial maps
before we can apply the above result. This is easy to do by just choosing a
subdivision $\tau$ of $F$ into quadrilaterals, and then replacing each $f_{s}$
by the unique simplicial map which agrees with $f_{s}$ on the vertices of
$\tau$.

Theorem~\ref{monotone} generalizes from $1$--parameter to $k$--parameter
families certain results of Thurston \cite{Thurston}, Bonahon \cite{Bonahon},
Canary \cite{Canary}, Minsky \cite{Minsky:1992} and Wolf \cite{Wolf}, that
utilized pleated surfaces and simplicial hyperbolic surfaces to control the
area and diameter of a $1$--parameter family of surfaces. See also the recent
use of bounded area $1$--parameter families of surfaces by Agol \cite{Agol}
and Calegari-Gabai \cite{CalegariGabai} in proving the tameness conjecture for
ends of negatively curved $3$--manifolds. Area bounds for $2$--parameter
families of surfaces obtained using smooth harmonic maps were used to find
counterexamples to the stabilization conjecture for Heegaard splittings in
\cite{HassThompsonThurston}. See the discussion at the end of this section.

Pitts and Rubinstein \cite{PittsRubinstein} gave a construction of unstable
minimal surfaces that utilizes a minimax argument applied to a family of
surfaces called a sweepout. For later applications we give below a modified
definition of a sweepout.

Given a closed orientable surface $F$, a closed orientable $3$--manifold $M$,
and a map $h:(F\times I,F\times\partial I)\rightarrow M$ with the property
that $\mbox{Area}(h(F,0))=\mbox{Area}(h(F,1))=0$, there is a natural
homomorphism
\[
\phi_{h}:H_{3}(F\times I,F\times\partial I)\rightarrow H_{3}(M).
\]
We now explain the construction of $\phi_{h}$.

Since $H_{3}(F\times I,F\times\partial I)$ is a cyclic group generated by the
fundamental class $\mu$ of $(F\times I,F\times\partial I)$, it suffices to
define the action of $\phi_{h}$ on $\mu$.

Let $\alpha$ denote a fundamental $3$--chain for $F\times I$, so that
$\partial\alpha=\beta_{1}-\beta_{0}$, where $\beta_{i}$ denotes a fundamental
$2$--cycle for $F\times\{i\}$, $i=0,1$. Thus each $h_{\#}(\beta_{i})$ is a
singular $2$--cycle in $M$ with zero volume. A theorem of Federer implies
that, for $i=0,1$, there is a singular $3$--chain $c_{i}$ in $M$ such that
$\partial c_{i}=h_{\#}(\beta_{i})$, and $c_{i}$ also has zero volume
\cite{Federer}. Thus the $3$--chain $c_{0}+h_{\#}(\alpha)-c_{1}$ is a cycle.
We want to define the homology class represented by this cycle to be $\phi
_{h}(\mu)$. We need to know that this is independent of the choices of the
$c_{i}$'s. Suppose that $c_{0}$ and $c_{0}^{\prime}$ are zero volume singular
$3$--chains in $M$ such that $\partial c_{0}=\partial c_{0}^{\prime}%
=h_{\#}(\beta_{i})$. Then $c_{0}-c_{0}^{\prime}$ is a zero volume $3$--cycle
in $M$, and a theorem of Federer implies this cycle must be null homologous.
It follows that our definition of $\phi_{h}(\mu)$ is independent of the choice
of $3$--chain $c_{0}$. Similarly it is also independent of the choice of
$3$--chain $c_{1}$. It follows that $\phi_{h}(\mu)$ is well defined.

\begin{definition}
Let $M$ be a closed orientable $3$--manifold, and let $F$ be a closed
orientable surface. A \emph{sweepout} of $M$ by $F$ is a map $h:F\times
I\rightarrow M$ with $\mbox{Area}(g(F,0))=\mbox{Area}(g(F,1))=0$, and inducing
an isomorphism
\[
\phi_{h}:H_{3}(F\times I,F\times\partial I)\rightarrow H_{3}(M).
\]

\end{definition}

\begin{remark}
One can equally well define a sweepout using homology with $\mathbb{Z}_{2}$ coefficients.
\end{remark}

There is a natural way to associate a sweepout to a Heegaard splitting of a
closed orientable $3$--manifold $M$. In a handlebody $H$, we can choose a
homotopy which shrinks the boundary surface $F$ onto a $1$--dimensional spine
$\Gamma$ of the handlebody, so that the homotopy is through embeddings except
at the end of the homotopy. Doing this for each of the two handlebodies of the
Heegaard splitting of $M$ yields a sweepout in which the surface
$F\allowbreak$ is the Heegaard surface and, as $t\rightarrow0$, the homotopy
crushes $F$ down to the chosen spine of one handlebody, and, as $t\rightarrow
1$, the homotopy crushes $F$ down to the chosen spine of the other handlebody.
Pitts and Rubinstein \cite{PittsRubinstein} considered a closed $3$--manifold
with a strongly irreducible Heegaard splitting. They then constructed a
minimal surface that has maximal area in some sweepout derived from the
splitting. In a hyperbolic manifold $M$, a minimal surface of genus $g$ has
area less than $2\pi(2g-2)$, and this implies the same area bound for each
surface in the sweepout.

The existence of a sweepout by bounded area surfaces has implications for the
geometry of $M$. It gives upper bounds on the injectivity radius, as noted in
Rubinstein \cite{Rubinstein} and in Bachman-Cooper-White \cite{BCW}.

We will show that one can use harmonic maps to obtain the same area bound
given by Pitts and Rubinstein, but without proving the existence of an
unstable minimal surface and without needing to assume that the Heegaard
splitting is strongly irreducible. The minimal surface obtained by Pitts and
Rubinstein is embedded, whereas harmonic maps need not yield embedded
surfaces, but embeddedness is not needed to establish results such as a bound
on injectivity radius.

The natural way to apply the theory of harmonic maps to a sweepout $h:F\times
I\rightarrow M$ is to apply Theorem \ref{monotone} or its smooth analogue in
order to obtain a sweepout by bounded area surfaces. In the smooth case the
difficulty is that at the ends of the sweepout, i.e. at $t=0$ and $t=1$, the
maps from $F$ to $M$ have $1$--dimensional image and so there is no way to
define the induced Riemannian metric on $F$ and one cannot directly use the
gradient flow. In \cite{HassThompsonThurston}, the authors showed how to work
round this difficulty. In the simplicial setting, this is not a problem, but a
new difficulty arises. It is not automatic that the maps from $F$ to $M$ at
the ends of the sweepout have zero energy. If they do not, then the gradient
flow applied to such maps may yield maps of non-zero Riemannian area, so that
we no longer have a sweepout. We resolve this difficulty in the following way.

First we consider maps from surfaces to graphs in general. Let $q:S^{1}\times
I\rightarrow I$ denote projection onto the second factor. Let $F$ be a closed
surface, and let $\Gamma$ be a graph.

\begin{definition}
A map $f:F\rightarrow\Gamma$ is \emph{of standard type} if there is a compact
subsurface $N$ of $F$, such that the following conditions hold:

\begin{enumerate}
\item Each component of $N$ is an annulus.

\item Each component of the closure of $M-N$ is sent by $f$ to a vertex of
$\Gamma$.

\item Each component annulus $A$ of $N$ is sent by $f$ to an edge of $\Gamma$,
and $f\mid A=S^{1}\times I$ is the composite of $q$ with the attaching map of
the edge.
\end{enumerate}
\end{definition}

Any map from $F$ to a graph $\Gamma$ is homotopic to a map of standard type.
For future reference, we set out this result and its proof.

\begin{lemma}
\label{standardtypeforonemap}Let $F$ be a closed surface, and let $\Gamma$ be
a graph. Then any map from $F$ to $\Gamma$ is homotopic to a map of standard type.
\end{lemma}

\begin{proof}
We denote the edges of $\Gamma$ by $e_{\lambda}$, $\lambda\in\Lambda$. The
first step is to choose a point $x_{\lambda}$ in the interior of each edge
$e_{\lambda}$ of $\Gamma$, and to homotope the map $f:F\rightarrow\Gamma$ to
be transverse to each $x_{\lambda}$. Let $C_{\lambda}$ denote $f^{-1}%
(x_{\lambda})$, so that $C_{\lambda}$ is a finite collection of disjoint
simple closed curves in $F$. Further there is a closed interval $E_{\lambda}$
in the interior of $e_{\lambda}$, and a homeomorphism of $f^{-1}(E_{\lambda})$
with $C_{\lambda}\times I$, such that the restriction of $f$ to each
$S^{1}\times I$ component of $f^{-1}(E_{\lambda})$ is the composite of $q$,
projection of $S^{1}\times I$ onto the second factor, with a homeomorphism
onto $E_{\lambda}$. Now consider a map $\phi:\Gamma\rightarrow\Gamma$ which,
for each edge $e_{\lambda}$ of $\Gamma$, maps $E_{\lambda}$ to $e_{\lambda}$
by a homeomorphism on their interiors and maps each component of $\Gamma-\cup
E_{\lambda}$ to the vertex of $\Gamma$ in that component. The composite
$\phi\circ f:F\rightarrow\Gamma$ is of standard type. As $\phi$ is clearly
homotopic to the identity, it follows that there is a homotopy of $f$ to a map
of standard type, as required.
\end{proof}

Next we consider constructing subdivisions into quadrilaterals. Pick a
subdivision of the circle $S^{1}$ into intervals. Taking the product with $I$
yields a subdivision of the annulus $S^{1}\times I$ into quadrilaterals, which
we call a standard subdivision.

Now given a map $f:F\rightarrow\Gamma$ of standard type, we choose a
subdivision $\tau$ of $F$ into quadrilaterals by starting with a standard
subdivision of each component of $N$, and then extending in any way to the
rest of $F$. We will say that such a subdivision is \textit{compatible} with
$f$.

Next let $l$ be the simplicial quasi-metric on $(F,\tau)$ induced by $f$, and
suppose that $\Gamma$ has a path metric. Thus we can consider the simplicial
energy $E_{S}(f)$ of $f:(F,\tau,l)\rightarrow\Gamma$. Let $P$ be a
quadrilateral of $\tau$. If $P$ is not contained in $N$, it is mapped by $f$
to a point and so contributes zero to $E_{S}(f)$. If $P$ is contained in $N$,
two of its edges are mapped to vertices of $\Gamma$ and the other two edges
are mapped onto edges of $\Gamma$. Further the two edges mapped to vertices
are non-adjacent edges of $P$. It follows that in this case also $P$
contributes zero to $E_{S}(f)$. We conclude that $E_{S}(f)=0$.

Before returning to our discussion of sweepouts, we need one further
refinement of Lemma \ref{standardtypeforonemap}. We consider a closed surface
$F$, a finite collection of graphs $\Gamma_{1},\ldots,\Gamma_{n}$, and maps
$f_{i}:F\rightarrow\Gamma_{i}$. Lemma \ref{standardtypeforonemap} tells us
that, for each $i$, we can homotope $f_{i}$ to be of standard type and can
then find a compatible subdivision $\tau_{i}$ of $F$ into quadrilaterals. The
following lemma shows how to find a single subdivision of $F$ into
quadrilaterals which is compatible with each $f_{i}$.

\begin{lemma}
\label{standardtypeforseveralmaps}Let $F$ be a closed surface, and, for $1\leq
i\leq n$, let $\Gamma_{i}$ be a graph, and let $f_{i}:F\rightarrow\Gamma_{i}$
be a continuous map. Then there is $g_{i}$ homotopic to $f_{i}$, and a single
subdivision of $F$ into quadrilaterals which is compatible with each $g_{i}$.
\end{lemma}

\begin{proof}
By applying Lemma \ref{standardtypeforonemap}, we can assume that each $f_{i}$
is of standard type. Choose one point in the interior of each edge of
$\Gamma_{i}$, and let $X_{i}$ denote the union of all these points. Thus
$f_{i}^{-1}(X_{i})$ is a finite collection of disjoint simple closed curves on
$F$. We can slightly homotope each $f_{i}$ to a new map $g_{i}$, also of
standard type, to arrange that the $g_{i}^{-1}(X_{i})$'s intersect each other
transversely and that there are no triple points. We let $C_{i}=g_{i}%
^{-1}(X_{i})$. For each $i$, we choose a regular neighborhood $N_{i}$ of
$C_{i}$ so that the following conditions hold:

\begin{enumerate}
\item The intersection of three distinct $N_{i}$'s is empty.

\item For distinct $i$ and $j$, the boundaries of $N_{i}$ and $N_{j}$ meet transversely.

\item For distinct $i$ and $j$, each component of $N_{i}\cap N_{j}$ is a disc
which contains exactly one point of $C_{i}\cap C_{j}$, exactly one sub-arc of
$C_{i}$ and exactly one sub-arc of $C_{j}$. The picture must be as shown in
Figure~\ref{strips}.
\end{enumerate}

\begin{figure}[ptbh]
\centering
\includegraphics[width=2in]{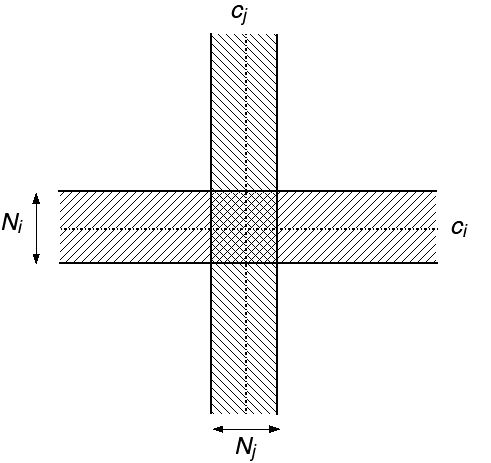} \caption{Intersection of $N_{i}$ and
$N_{j}$ }%
\label{strips}%
\end{figure}

Now we can construct a single subdivision $\tau$ of $F$ into quadrilaterals
which is simultaneously compatible with each $f_{i}$. For distinct $i$ and
$j$, consider a component of $N_{i}\cap N_{j}$. As in Figure~\ref{strips},
there is a natural way to regard this as a quadrilateral, with two edges in
$\partial N_{i}$ and two in $\partial N_{j}$. We start by choosing all these
quadrilaterals. Next we consider a component of $N_{i}$. This is an annulus,
and some pieces of this annulus have already been declared to be
quadrilaterals. It is now trivial to subdivide the rest so as to obtain a
standard subdivision of the annulus. Having made these choices for each $i$,
we will have a decomposition of the union of all the $N_{i}$'s. We now
complete this to a subdivision of $F$ in any way. It is easy to see that,
after a further homotopy, $\tau$ is compatible with each $g_{i}$, as required.
\end{proof}

Next we apply Theorem \ref{monotone} to sweepouts.

\begin{lemma}
\label{boundedarea} If $M$ is a closed orientable hyperbolic $3$--manifold
with a genus $g$ Heegaard splitting, then $M$ has a sweepout by genus $g$
surfaces with Riemannian area bounded above by $2\pi(2g-2)$.
\end{lemma}

\begin{proof}
Let $F$ denote the Heegaard surface of $M$, and let $H_{0}$ and $H_{1}$ denote
the two handlebodies into which $F$ cuts $M$. For $i=0,1$, pick a
$1$--dimensional spine $\Gamma_{i}$ for $H_{i}$. By subdividing the edges of
these graphs, we can arrange that each edge is a geodesic segment in $M$. Then
there is a sweepout $g:F\times I\rightarrow M$ of $M$ by $F$ such that $g_{0}$
collapses $F$ to $\Gamma_{0}$, and $g_{1}$ collapses $F$ to $\Gamma_{1}$. For
each $i=0$ or $1$, there is a map $f_{i}:F\rightarrow\Gamma_{i}$, which is of
standard type, and is homotopic to $g_{i}$. We homotope the given sweepout $g$
so that $g_{0}=f_{0}$ and $g_{1}=f_{1}$. The above discussion yields a
subdivision $\tau$ of $F$ into quadrilaterals which is compatible with both
$g_{0}$ and $g_{1}$, perhaps after further small homotopies of $g_{0}$ and
$g_{1}$. As $g_{0}$ and $g_{1}$ are now of standard type, and the edges of
$\Gamma_{0}$ and $\Gamma_{1}$ are geodesic, it follows that $g_{0}$ and
$g_{1}$ are simplicial. Now we simultaneously homotope each $g_{t}$ to the
simplicial map which agrees with $g_{t}$ on the vertices. This will not alter
$g_{0}$ and $g_{1}$, so these maps still have zero simplicial energy with
respect to the induced simplicial quasi-metric. Finally we can apply the proof
of Theorem \ref{monotone} to this sweepout by simplicial maps to obtain a
$1$--parameter family of simplicial harmonic maps. The crucial fact is that as
the energies of $g_{0}$ and $g_{1}$ are each zero, the gradient flow leaves
them unchanged. Thus we have homotoped the original sweepout, through
sweepouts, to a sweepout by simplicial harmonic maps, each of which has
Riemannian area less than $2\pi(2g-2)$, by Lemma \ref{E>standardA}.
\end{proof}

The Scharlemann-Thompson genus of $M$ can be used in place of the Heegaard
genus $g$ in Corollary~\ref{boundedarea}, and may be smaller
\cite{ScharlemannThompson:94}.

We end this section by discussing how to replace the use of smooth harmonic
maps by simplicial harmonic maps in the work of Hass, Thompson and Thurston
\cite{HassThompsonThurston}. In that paper, the authors considered a
particular closed hyperbolic $3$--manifold $M$ with a Heegaard splitting
$M=H_{0}\cup H_{1}$ and Heegaard surface $F$. They showed that there could not
be an ambient isotopy of $M$ which interchanged $H_{0}$ and $H_{1}$, and a key
role in their argument was played by applying the gradient flow to a
$2$--parameter family of smooth maps from $F$ to $M$. We will describe how to
replace this part of their argument by the work in this paper.

As in the proof of Lemma \ref{boundedarea}, for $i=0,1$, we pick a
$1$--dimensional spine $\Gamma_{i}$ for $H_{i}$, so that each edge of
$\Gamma_{i}$ is a geodesic segment in $M$. As in that lemma, we can find a
sweepout $h:F\times I\rightarrow M$ by simplicial harmonic maps such that, for
$i=0,1$, the map $h_{i}$ sends $F$ to $\Gamma_{i}$ and is of standard type. In
particular $h_{0}$ and $h_{1}$ each has zero simplicial energy.

Next suppose that $\Phi:M\times I\rightarrow M\times I$ is an ambient isotopy
of $M$ which interchanges $H_{0}$ and $H_{1}$. Thus $\Phi$ is a homeomorphism,
$\Phi_{0}$ is the identity map of $M$, and $\Phi_{1}(H_{0})=H_{1}$ and
$\Phi_{1}(H_{1})=H_{0}$. Let $p:M\times I\rightarrow M$ denote projection.
Then the composite map $F\times I\times I\overset{h\times1}{\rightarrow
}M\times I\overset{\Phi}{\rightarrow}M\times I\overset{p}{\rightarrow}M$ is a
$2$--parameter family of maps $h_{s,t}$ from $F$ to $M$, with $h_{s,0}=h_{s}$,
for all $s$. As $h_{0}$ collapses $F$ to the spine $\Gamma_{0}$ of $H_{0}$,
the map $h_{0,t}$ collapses $F$ to the graph $\Phi_{t}(\Gamma_{0})$ embedded
in $M$. Further each map $h_{0,t}$ is of standard type. Similarly each map
$h_{1,t}$ collapses $F$ to the graph $\Phi_{t}(\Gamma_{1})$ embedded in $M$,
and is of standard type. Note that as $\Phi_{1}$ interchanges $H_{0}$ and
$H_{1}$, we have $\Phi_{1}(\Gamma_{0})$ is a spine of $H_{1}$ and $\Phi
_{1}(\Gamma_{1})$ is a spine of $H_{0}$.

The next step is to replace each map $h_{s,t}$ by the unique simplicial map
which agrees with $h_{s,t}$ on the vertices of $\tau$. As each $h_{s}$ is
already simplicial, the maps $h_{s,0}=h_{s}$ do not change.\ The maps
$h_{0,t}$ and $h_{1,t}$ may change, but they will continue to have zero
simplicial energy. This is because any edge of $\tau$ which was mapped to a
point before the change will continue to be mapped to a point after the
change. Thus the gradient flow will leave all the maps $h_{0,t}$ and $h_{1,t}$
unchanged. Hence when we apply Theorem \ref{monotone} to this $2$--parameter
family of maps, we will obtain a family $f_{s,t}$ of harmonic maps, each with
Riemannian area less than $2\pi\left\vert \chi(F)\right\vert $, such that all
the maps $f_{0,t}$ and $f_{1,t}$ have zero simplicial energy and hence zero
Riemannian area.

In \cite{HassThompsonThurston}, the authors considered hyperbolic
$3$--manifolds of fixed Heegaard genus which have arbitrarily high volume.
This implies that a sweepout $h:F\times I\rightarrow M$ of $M$ by $F$ in which
each surface has Riemannian area less than $2\pi\left\vert \chi(F)\right\vert
$, must be ``long'' in a certain sense. They show that this implies that in
any $2$--parameter family that flips the orientation of the Heegaard
splitting, there must be a surface of large area. This in turn was used to
show that it is not possible to connect such a Heegaard sweepout to the same
sweepout with opposite orientation without stabilizing until the genus is doubled.

\section{Sweepouts in higher dimensions}

\label{genus}

In this section we define the genus of an $n$--dimensional closed orientable
manifold, extending the idea of the Heegaard genus of a $3$--manifold to
closed manifolds of arbitrary dimension. To do so we extend the definition of
a sweepout of a $3$--manifold by a $1$--parameter family of surfaces to a
sweepout of a closed $n$--dimensional manifold $M$ by an $(n-2)$--parameter
family of surfaces. We replace the unit interval $I$ in a Heegaard sweepout
with an arbitrary compact $(n-2)$--dimensional manifold which parametrizes the
surface maps.

For this construction we let $X$ be the orientable total space of a bundle
over a compact $(n-2)$--manifold $B$ with fiber a closed orientable surface
$F$, with projection map $\pi:X\rightarrow B$, and with $\partial B$ possibly
empty. We will consider an appropriately defined degree one map
$h:X\rightarrow M$ such that, for all $b\in\partial B$,%
\[
Area(h(\pi^{-1}(b)))=0.
\]
As before, we need to define a homomorphism
\[
\phi_{h}:H_{n}(X,\partial X)\rightarrow H_{n}(M)
\]
to make sense of the notion of degree for such maps. Again note that we could
also use homology with $\mathbb{Z}_{2}$ coefficients.

Since $H_{n}(X,\partial X)$ is a cyclic group generated by the relative
fundamental class $\mu$, a homomorphism $\phi_{h}$ to $H_{n}(M)$ is determined
by the action of $\phi_{h}$ on $\mu$. So long as $h$ satisfies some mild
smoothness condition, the $(n-1)$--volume of $h(\partial X)$ will equal zero
because $Area(h(\pi^{-1}(b)))=0$ for $b\in\partial B$. So $h(\partial X)$
gives a zero volume $(n-1)$--chain in $M$ which in turn is the boundary of a
zero volume $n$--chain $c$ in $M$. Thus $h_{\#}(\mu)+c$ is a $n$--cycle in $M$
whose homology class $\alpha\in H_{n}(M)$ is independent of the choice of $c$.
We set $\phi_{h}(\mu)=\alpha$. This leads to the following definition.

\begin{definition}
A \emph{sweepout} of a closed orientable $n$--manifold $M$ by a closed
orientable surface $F$ is a map $h:X\rightarrow M$ from a bundle $X$ with
orientable total space, fiber $F$, and base space a compact $(n-2)$--manifold
$B$, whose boundary may be empty, such that

\begin{enumerate}
\item $Area(h(\pi^{-1}(b)))=0$, for all $b\in\partial B$, and the
$(n-1)$--volume of $h(\partial X)$ also equals zero, and

\item $h$ induces an isomorphism%
\[
\phi_{h}:H_{n}(X,\partial X)\rightarrow H_{n}(M).
\]

\end{enumerate}
\end{definition}

\begin{definition}
If $M$ is a closed orientable $n$--manifold which admits a sweepout by genus
$g$ surfaces, for some $g$, then the \emph{genus} of $M$ is the smallest such
number $g$.
\end{definition}

For $3$--manifolds this number is clearly less than or equal to the Heegaard
genus, since the above definition allows non-embedded and non-disjoint
surfaces to sweep over the $3$--manifold. In addition, it allows the
consideration of sweepouts by bundles over a circle, as well as over an
interval. We do not know whether or not equality always holds for a given $B$.

Here are some examples of $n$--manifolds which admit a sweepout by surfaces.
Our first examples are simply the product of a $3$--manifold with another
manifold. Let $Y$ be a closed $(n-3)$--dimensional manifold, let $Z$ be a
closed $3$--manifold, and let $M=Y\times Z$. Choose a sweepout $h:I\times
F\rightarrow Z$ of $Z$ by a closed surface $F$, let $B=Y\times I$, and let
$X=B\times F$. Then the map $X=B\times F=Y\times(I\times F)\overset
{(1,h)}{\rightarrow}Y\times Z=M$ is a sweepout of $M$ by $F$.

Our next examples will play an important role in our discussion below. Let $B$
be a compact $(n-2)$--dimensional manifold, and let $F$ be a closed orientable
surface of genus $g$. Let $X$ be a bundle over $B$ with fiber $F$ whose
restriction to $\partial B$ is trivial. Choose a homeomorphism $f$ from $F$ to
the boundary of a handlebody $H$, and let $g$ denote the homeomorphism
$f\times1:F\times\partial B\rightarrow\partial H\times\partial B$. We form a
closed $n$--manifold $M$ by attaching $X$ to $H\times\partial B$ using the
homeomorphism of their boundaries $\partial X=F\times\partial B\overset
{g}{\rightarrow}\partial H\times\partial B=\partial(H\times\partial B)$. As in
the previous section, there is a homotopy which shrinks the boundary surface
$F$ of the handlebody $H$ onto a $1$--dimensional spine $\Gamma$ of $H$, such
that the homotopy is through embeddings except at the end of the homotopy.
Taking the product of this homotopy with the identity map of $B$ yields a map
$(I\times\partial B)\times F\rightarrow\partial B\times H$ which, for each $b$
in $\partial B$, shrinks $\{b\}\times F$ onto $\{b\}\times\Gamma$. Together
with the identity map of $X$, this yields a map from $X\cup((I\times\partial
B)\times F)$ to $M$ which is clearly a sweepout of $M$ by $F$. Of course
$X\cup((I\times\partial B)\times F)$ is homeomorphic to $X$, so we can and
will regard this sweepout as a map from $X$ to $M$. This sweepout has the
special property that for all interior points of $B$, the associated map from
$F$ to $M$ is an embedding, and that for all boundary points of $B$, the
associated map from $F$ to $M$ has image a copy of the graph $\Gamma$. If
$\partial B$ is not connected, one can make this construction slightly more
general by separately choosing a homeomorphism from $F$ to the boundary of the
handlebody $H$ for each component of $\partial B$.

When $B$ is the $(n-2)$-ball, the bundle must be trivial over $\partial B$.
These \emph{$n$-dimensional Heegaard splittings} therefore fall into this
class of examples.

Now suppose that we have a sweepout $h:X\rightarrow M$ of a closed orientable
$n$--manifold $M$ by a closed orientable surface $F$. In addition suppose that
$M$ has a metric of negative curvature. We would like to apply the ideas of
the previous section to show that $M$ admits a sweepout by small area maps
from $F$ to $M$. As in that section, the key to doing this is being able to
find a sweepout such that, for each $b$ in $\partial B$, the associated map
from $F$ to $M$ has zero energy. The result we obtain is the following.

\begin{theorem}
Let $M$ be a closed Riemannian $n$--manifold with sectional curvatures bounded
above by $-a<0$, let $F$ be a closed orientable surface of genus $g\geq2$, and
let $\tau$ be a subdivision of $F$ into quadrilaterals. Suppose that
$h:X\rightarrow M$ is a sweepout of $M$ by $F$ such that, for each $b$ in
$\partial B$, the associated map $F\rightarrow M$ is simplicial and has zero
simplicial energy with respect to the induced quasi-metric on $(F,\tau)$. Then
$h$ can be homotoped, fixing $\partial X$, to a new sweepout so that, for each
$b$ in $B$, the associated map from $F$ to $M$ has Riemannian area at most
$2\pi(2g-2)/a$.
\end{theorem}

\begin{proof}
We will use the same ideas as in the proof of Lemma \ref{boundedarea} in the
case of dimension $3$.

If $X$ is a trivial bundle over $B$, we can simply apply
Theorem~\ref{monotone} to homotope the sweepout to a family of simplicial
harmonic maps from $F$ to $M$. This homotopy will be fixed for fibers over
$\partial B$, so still gives a degree one map from $X$ to $M$, and is
therefore still a sweepout. Theorem~\ref{monotone} also tells us that in this
new sweepout the area of each surface is less than $2\pi(2g-2)/a$, as required.

If $X$ is a nontrivial bundle over $B$, it is still locally trivial, and that
suffices to apply Theorem~\ref{monotone}.
\end{proof}

Given a sweepout, it is not obvious that one can find a subdivision into
quadrilaterals that satisfies the hypotheses of the above theorem. But it is
easy to do this in the cases discussed just before the theorem. As in the
second set of examples, we let $B$ be a compact $(n-2)$--dimensional manifold,
let $F$ be a closed orientable surface of genus $g$, and let $X$ be a bundle
over $B$ with fiber $F$ whose restriction to $\partial B$ is trivial. Finally
let $M$ be a closed $n$--manifold formed by attaching $X$ to $H\times\partial
B$ using a product homeomorphism of their boundaries. Then there is a sweepout
$h:X\rightarrow M$ such that, for all interior points of $B$, the associated
map from $F$ to $M$ is an embedding, and, for all boundary points of $B$, the
associated map from $F$ to $M$ has image a graph which depends only on the
component of $\partial B$. Further, for a component $C$ of $\partial B$, the
maps from $F$ to the graph $\Gamma$ associated to the points of $C$ are all
equal. As $\partial B$ has only finitely many components, Lemma
\ref{standardtypeforseveralmaps} tells us that after a suitable homotopy of
$h$, there is a single subdivision $\tau$ of $F$ into quadrilaterals which is
compatible with all of these maps. Now it is immediate that, for each $b$ in
$\partial B$, the associated map from $F$ to $M$ has zero simplicial energy
with respect to the induced quasi-metric on $(F,\tau)$. If we assume that $M$
has negative curvature, then we can homotope this sweepout to the sweepout by
simplicial maps which agrees with the original sweepout on the vertices of
$\tau$. Thus, for each $b$ in $\partial B$, the associated simplicial map from
$F$ to $M$ has zero simplicial energy with respect to the induced quasi-metric
on $(F,\tau)$, verifying all the hypotheses of the above theorem.

We can generalize the previous discussion as follows. Let $M$ be as in the
previous paragraph, but do not assume that $M$ has negative curvature. Suppose
that there a map $k$ of degree $1$ from $M$ to some negatively curved
$n$--manifold $M^{\prime}$. Then the composite map $k\circ h:X\rightarrow
M^{\prime}$ is a sweepout such that, for each $b$ in $\partial B$, the
associated map from $(F,\tau,l)$ to $M^{\prime}$ has zero simplicial energy.
Further, we can homotope this sweepout to the sweepout by simplicial maps
which agrees with the original sweepout on the vertices of $\tau$, again
verifying all the hypotheses of the above theorem.

\section{Varying the simplicial metric}

\label{limits}

In this section we examine global minimizers of simplicial energy over the
space of all simplicial metrics for a fixed triangulation. It was known to
Courant and Rado that one way to find a least area map in a homotopy class is
to find a map that minimizes energy not just for a given metric on the domain
but among all possible metrics \cite{Courant, Rado}. We now present this
argument in the simplicial setting, where it becomes considerably easier.

\begin{lemma}
Let $f:F\rightarrow M$ be a map of a triangulated closed surface $F$ to a
closed manifold $M$ with a Riemannian metric. Let $E_{S}(f,l)$ denote the
simplicial energy of $f$ with respect to a simplicial quasi-metric $l$ on $F$.

\begin{enumerate}
\item If $A_{S}(f)$ has least possible value among all maps homotopic to $f$,
then there is a simplicial quasi-metric $l$ on $F$ such that $E_{S}(f,l)$ has
the least possible value among all maps homotopic to $f$ and over all
simplicial quasi-metrics on $F$. Further $E_{S}(f,l)=A_{S}(f)$.

\item If $l$ is a simplicial quasi-metric on $F$, and $f$ is a map such that
$E_{S}(f,l)$ has least possible value among all maps homotopic to $f$ and over
all simplicial quasi-metrics on $F$, then $A_{S}(f)$ also has the least
possible value, and $E_{S}(f,l)=A_{S}(f)$.
\end{enumerate}
\end{lemma}

\begin{proof}
(1) Let $l$ denote the simplicial quasi-metric on $(F,\tau)$ induced by $f$,
i.e. for each edge $e_{i}$ of $\tau$ , we set $l_{i}=L_{i}$. For this
quasi-metric, all the stretch factors $\sigma_{i}$ are equal to $1$ on all
edges with $l_{i}>0$ and therefore $E_{S}(f,l)=A_{S}(f)$, by Lemma \ref{E>A}.
The hypothesis that $A_{S}(f)$ has least possible value implies that, for any
map $f^{\prime}$ homotopic to $f$, we have $A_{S}(f^{\prime})\geq A_{S}(f)$.
Thus for any simplicial quasi-metric $k$ on $F$, we have the inequalities
\[
E_{S}(f^{\prime},k)\geq A_{S}(f^{\prime})\geq A_{S}(f)=E_{S}(f,l).
\]
It follows that $E_{S}(f,l)$ has the least possible value over all maps
homotopic to $f$ and over all simplicial quasi-metrics on $F$, and that
$E_{S}(f,l)=A_{S}(f)$, as claimed.

(2) Suppose that $A_{S}(f)$ does not have the least possible value. Thus there
is a map $f^{\prime}$ homotopic to $f$ with $A_{S}(f^{\prime})<A_{S}(f)$. We
define a simplicial quasi-metric $k$ on $F$ by setting $k_{i}$ equal to
$L_{i}^{\prime}$. For this quasi-metric and the map $f^{\prime}$, the stretch
factors $\sigma_{i}$ are equal to $1$ on all edges with $l_{i}>0$. Thus, by
Lemma \ref{E>A}, we have
\[
E_{S}(f^{\prime},k)=A_{S}(f^{\prime})<A_{S}(f)\leq E_{S}(f,l)
\]
which contradicts the hypothesis that $E_{S}(f,l)$ has least possible value.

We conclude that $A_{S}(f)$ must have the least possible value. Now part (1)
implies that there is a simplicial quasi-metric $k$ on $F$ such that
$E_{S}(f,k)$ has the least possible value among all maps homotopic to $f$ and
over all simplicial quasi-metrics on $F$. Further $E_{S}(f,k)=A_{S}(f)$. As
$E_{S}(f,l)$ also has least possible value, it follows that $E_{S}(f,l)$ must
equal $E_{S}(f,k)$, so that $E_{S}(f,l)=A_{S}(f)$, as required.
\end{proof}

The proof of the above lemma is much simpler than in the smooth setting. This
is because any map of $F$ into $M$ for which one can measure the lengths of
the edges of $\tau$ induces a simplicial quasi-metric on $F$, whereas in the
smooth setting, one needs $f$ to be an immersion in order to induce a
Riemannian metric on $F$.

In Section~\ref{simplicialharmonic} we saw that a finite energy map from a
surface with a fixed simplicial metric or quasi-metric to a manifold of
non-positive curvature can be homotoped to a simplicial harmonic map that
minimizes simplicial energy. We now examine what happens when the simplicial
metric on the domain surface is allowed to change. We will show that the class
of maps homotopic to a given map $f:F\rightarrow M$ contains a map with
minimal simplicial energy among all maps homotopic to $f$ and all possible
simplicial quasi-metrics $(F,\tau,l)$. In the smooth setting this gives a
least area map in the homotopy class of $f$. The existence of a least area map
in that setting was shown by Sacks and Uhlenbeck \cite{SacksUhlenbeck} and by
Schoen and Yau \cite{SchoenYau:78}, with the hypotheses that the surface is
orientable with genus at least $1$, and its fundamental group injects into the
fundamental group of $M$.

It would seem more natural to consider only simplicial metrics rather than
allow quasi-metrics, but the following discussion explains why this is
unreasonable. In fact this is why we introduced the idea of a simplicial
quasi-metric. Consider a map $f_{1}:F\rightarrow M$ from a triangulated
surface $(F,\tau)$ to a Riemannian manifold and let
\[
\mathcal{I}=\inf\{E_{S}(f)\}
\]
where the infimum is taken over all maps $f$ homotopic to $f_{1}$ and over all
simplicial metrics on $\tau$. We would like to establish the existence of a
simplicial harmonic map homotopic to $f_{1}$ with simplicial energy equal to
$\mathcal{I}$. There is a sequence $l^{n}$ of simplicial metrics on $F$, and
of homotopic simplicial maps $f_{n}:F\rightarrow M$ such that $E_{S}%
(f_{n})\rightarrow\mathcal{I}$, as $n\rightarrow\infty$. We call such a
sequence a \textit{simplicial-energy minimizing} sequence. The main difficulty
we face is that even if the sequence $l^{n}$ converges to a function $l^{0}$
on the edges of $\tau$, it may be that $l^{0}(e_{i})=0$, for some edge $e_{i}$
of $\tau$. It is clear that $l^{0}(e_{i})\geq0$, for all $i$, and that $l^{0}$
satisfies the triangle inequality for each triangle of $\tau$. Thus $l^{0}$ is
a simplicial quasi-metric on $F$. There is a further difficulty with this
approach which is that it is possible that $l^{0}(e_{i})=0$, for every edge
$e_{i}$ of $\tau$. For example, suppose that $l$ is a simplicial metric on
$F$, and $f:F\rightarrow M$ is a simplicial map such that $E_{S}%
(f)\rightarrow\mathcal{I}$. For each $n\geq1$, we define the simplicial metric
$l^{n}$ by $l^{n}(e_{i})=l(e_{i})/n$, for each edge $e_{i}$ of $\tau$, and the
simplicial map $f_{n}$ to equal $f$. The scale invariance of simplicial energy
implies that $E_{S}(f_{n})=\mathcal{I}$, for each $n\geq1$, but the sequence
$l^{n}$ converges to the simplicial quasi-metric $l^{0}$ such that
$l^{0}(e_{i})=0$, for every edge $e_{i}$ of $\tau$. The analogous phenomenon
occurs in the smooth setting, and as in that setting we avoid this problem by
changing our initial choice of minimizing sequence.

The following lemma shows that allowing quasi-metrics does not change the
minimal energy for a homotopy class.

\begin{lemma}
Let $f_{1}:F\rightarrow M$ be a map of a triangulated closed surface
$(F,\tau)$ to a closed Riemannian manifold $M$. Let $\mathcal{I}$ denote the
infimum of simplicial energies over the family $\mathcal{F}$ of all maps $f$
homotopic to $f_{1}$ and over all simplicial metrics on $\tau$, and let
$\mathcal{I}_{0}$ denote the infimum of simplicial energies over the family
$\mathcal{F}_{0}$ of all finite energy maps $f$ homotopic to $f_{1}$ and over
all simplicial quasi-metrics on $\tau$.

Then $\mathcal{I}=\mathcal{I}_{0}$.
\end{lemma}

\begin{proof}
As $\mathcal{F}\subset\mathcal{F}_{0}$, it is immediate that $\mathcal{I}%
\geq\mathcal{I}_{0}$.

Now let $l^{0}$ be a simplicial quasi-metric on $F$, and let $f:(F,\tau
,l^{0})\rightarrow M$ be a finite energy map homotopic to $f_{1}$. Consider
the family of edge length functions $l^{t}$, defined for $t>0$ by $l^{t}%
(e_{i})=l^{0}(e_{i})+t$. As $l^{t}$ is positive and satisfies all triangle
inequalities in $\tau$, it follows that $l^{t}$ is a simplicial metric for
each $t>0$.

Let $E_{S}(f,l^{t})$ denote the energy of the map $f:F\rightarrow M$ computed
using the simplicial quasi-metric $l^{t}$ on $\tau$. For $t>0$ we have
\[
E_{S}(f,l^{t})=\frac{1}{2}\sum\left(  \frac{L_{i}^{2}}{(l^{t}(e_{i}))^{2}%
}+\frac{L_{j}^{2}}{(l^{t}(e_{j}))^{2}}\right)  l^{t}(e_{i})l^{t}(e_{j}).
\]
If either of $l^{0}(e_{i})$ or $l^{0}(e_{j})$ is zero then so is $L_{i}$ or
$L_{j}$, so that
\[
\lim_{t\rightarrow0}E_{S}(f,l^{t})=E_{S}(f,l^{0}).
\]

It follows immediately that $\mathcal{I}\leq\mathcal{I}_{0}$, so that
$\mathcal{I}=\mathcal{I}_{0}$ as required.
\end{proof}

Next we show the existence of a simplicial quasi-metric and a map which
minimizes energy in a homotopy class. Note that, unlike the smooth case, this
theorem places no requirements on the given map $f_{1}$ or on its homotopy
class. In the smooth case, an incompressibility condition is required.

\begin{theorem}
\label{absolutelyleastenergymapexists}Let $f_{1}:F\rightarrow M$ be a map of a
triangulated closed surface $(F,\tau)$ to a closed non-positively curved
manifold $M$. Then there is a simplicial quasi-metric $l^{0}$ on $F$, and a
simplicial map $f_{0}$ homotopic to $f_{1}$, such that $f_{0}$ minimizes
simplicial energy among all maps homotopic to $f_{1}$ and all choices of
simplicial quasi-metric on $(F,\tau)$. Further $f_{0}$ minimizes simplicial
area among all maps homotopic to $f_{1}$.
\end{theorem}

\begin{proof}
Let $\mathcal{I}$ denote the infimum of simplicial energies over the family of
all finite energy maps $f$ homotopic to $f_{1}$ and over all simplicial
quasi-metrics on $\tau$.

Consider a sequence $l^{n}$ of simplicial quasi-metrics on $(F,\tau)$ and
finite-energy simplicial maps $f_{n}:F\rightarrow M$, each homotopic to $f$,
such that $\lim_{n\rightarrow\infty}E_{S}(f_{n})=\mathcal{I}$. For each $n$,
let $L_{i}^{n}$ denote the length of $f_{n}|e_{i}$ and define a new simplicial
quasi-metric $k^{n}$ on $(F,\tau)$ by $k^{n}(e_{i})=L_{i}^{n}$. For this
quasi-metric, the map $f_{n}$ is a simplicial isometry, with $l_{i}^{n}%
=L_{i}^{n}$ and $E_{S}(f_{n},k^{n})=A_{S}(f_{n})$. Now $A_{S}(f_{n})$ is
independent of the simplicial quasi-metric on $F$ and $A_{S}(f_{n})\leq
E_{S}(f_{n},l^{n})$ by Lemma~\ref{E>A}. Hence $E_{S}(f_{n},k^{n})\leq
E_{S}(f_{n},l^{n})$, and replacing each simplicial quasi-metric $l^{n}$ by
$k^{n}$ retains the property that $\lim_{n\rightarrow\infty}E_{S}%
(f_{n})=\mathcal{I}$, so that $f_{n}$ is still an energy minimizing sequence.
In particular, the sequences $E_{S}(f_{n})$ and $A_{S}(f_{n})$ are bounded.

We now claim that, after passing to a subsequence, the sequence $L_{i}^{n}$ is
bounded for each $i$. First note that by passing to a subsequence we can
arrange that each sequence $L_{i}^{n}$ is either convergent or unbounded. Now
consider an edge $e_{0}$ and suppose that $\{L_{0}^{n}\}$ is unbounded. By
passing to a further subsequence we can assume that $\{L_{0}^{n}\}$ is
increasing and unbounded. As $e_{0}$ meets one triangle of $\tau$ on each of
its two sides, there are four additional edges of $\tau$ which belong to a
triangle containing $e_{0}$. We denote these edges by $e_{1},e_{2},e_{3}%
,e_{4}$ (they may not all be distinct in $F$). The terms in $A_{S}(f_{n})$
which involve $L_{0}^{n}$ are $\sum_{i=1}^{4}L_{0}^{n}L_{i}^{n}=L_{0}^{n}%
\sum_{i=1}^{4}L_{i}^{n}$, and this sequence of sums is bounded, as we noted
above. As $L_{0}^{n}$ is increasing and unbounded, the sum $\sum_{i=1}%
^{4}L_{i}^{n}$ must approach $0$ as $n\rightarrow\infty$. Since each
$L_{i}^{n}\geq0$, we deduce that $L_{i}^{n}\rightarrow0$ as $n\rightarrow
\infty$, for $i=1,2,3,4$. Now the triangle inequality implies that $L_{0}^{n}$
must also converge to $0$ as $n\rightarrow\infty$, a contradiction. This
establishes that $L_{i}^{n}$ is bounded for each $i$, as claimed. It follows
that there is a subsequence of $f_{n}$ for which each sequence $L_{i}^{n}$ is
convergent. As $M$ is compact, we can find a further subsequence such that for
each vertex $v$ of $\tau$, the sequence $\{f_{n}(v)\}$ converges. As
$k^{n}(e_{i})=L_{i}^{n}$, the sequence $\{k^{n}\}$ also converges and the
limit is a simplicial quasi-metric $k^{0}$ on $F$. So the sequence of
finite-energy simplicial maps $f_{n}$ converges to a finite-energy simplicial
map $f_{0}$ with $E(f_{0})=\mathcal{I}$. This completes the proof of the theorem.
\end{proof}

We say that a map of a surface into $M$ is \emph{simply essential} if it sends
nontrivial simple loops on the surface to nontrivial elements of $\pi_{1}(M)$.
For two-sided embedded surfaces in a $3$--manifold this is equivalent to being
$\pi_{1}$--injective, but for singular surfaces in a $3$--manifold this
equivalence is unknown, and is called the simple loop conjecture. The result
below shows that if we assume that the map $f_{1}:F\rightarrow M$ in Theorem
\ref{absolutelyleastenergymapexists} is simply essential, then we can refine
the result to assert the existence of a simplicial metric (not just
quasi-metric) on $F$ and an energy minimizing map.

\begin{corollary}
\label{simplyessential} Let $f_{0}:F\rightarrow M$ be a simply essential
simplicial map from a triangulated closed surface $(F,\tau)$ to a closed
non-positively curved manifold $M$, and suppose that $l$ is a simplicial
quasi-metric on $(F,\tau)$ such that $E(f_{0})=\mathcal{I}$. Then there is a
new triangulation $\tau^{\prime}$ of $F$, obtained by collapsing edges and
triangles of $\tau$, a new simplicial map $f_{0}^{\prime}$ homotopic to
$f_{0}$, and a simplicial metric on $\tau^{\prime}$ for which $E(f_{0}%
^{\prime})=\mathcal{I}$.
\end{corollary}

\begin{proof}
As $f_{0}$ is simplicial, it maps any edge of $\tau$ to a point or to a
geodesic arc, and it maps any triangle of $\tau$ to a point or to a geodesic
arc or to a triangle. If an edge $e$ of $\tau$ is zero length, i.e. $l(e)=0$,
the fact that $f_{0}$ has finite energy implies that it must map $e$ to a
point. If a triangle $\sigma$ of $\tau$ has $\partial\sigma$ consisting
entirely of zero length edges, the fact that $f_{0}$ is simplicial implies
that it must map $\sigma$ to a point. If $\partial\sigma$ has two zero length
edges, the fact that $l$ satisfies the triangle inequality implies that all
three edges of $\partial\sigma$ are zero length. Finally if $\partial\sigma$
has one zero length edge, the fact that $l$ satisfies the triangle inequality
implies that the other two edges of $\partial\sigma$ have the same
$l$--length, and the fact that $f_{0}$ is simplicial implies that these edges
have the same image under $f_{0}$.

Now we let $\Delta$ denote the subcomplex of $\tau$ consisting of all zero
length edges of $\tau$ together with all triangles of $\tau$ whose boundary
consists of zero length edges. Thus $f_{0}$ maps each component of $\Delta$ to
a point. As $f_{0}$ is simply essential, if $N$ denotes a regular neighborhood
of a component of $\Delta$, all except one of the components of the closure of
$F-N$ must be a disc, and the union of $N$ with these discs must itself be a
disc. Suppose there is a disc component $B$ of the closure of $F-N$, and let
$R\ $denote the component of $F-\Delta$ which contains $B$. Then $\partial R$
is contained in $\Delta$, so is mapped to a point by $f_{0}$. The facts that
$f_{0}$ maps $\partial R$ to a point, is simplicial, and has least energy
implies that it must map all of $R$ to a point. Thus $R$ is contained in
$\Delta$, which is a contradiction. This shows that no component of the
closure of $F-N$ can be a disc, so that $N$ itself must be a disc.

We form the quotient $F^{\prime\prime}$ of $F$ by collapsing every simplex of
$\Delta$ to a point. Note that $f_{0}$ must factor through a map
$f_{0}^{\prime\prime}:F^{\prime\prime}\rightarrow M$. The preceding discussion
implies that $F^{\prime\prime}$ is homeomorphic to $F$. The image of $\tau$ in
$F^{\prime\prime}$ need not be a triangulation. For if $\sigma$ is a triangle
of $\tau$ such that $\partial\sigma$ has one zero length edge, then the image
of $\sigma$ in $F^{\prime\prime}$ is a $2$--gon. But the image of $\sigma$ in
$M$ is a geodesic arc, and so we can take a further quotient $F^{\prime}$ of
$F^{\prime\prime}$, by collapsing each such $2$--gon to an arc, and $f_{0}$
must factor through a map $f_{0}^{\prime}:F^{\prime}\rightarrow M$. Again
$F^{\prime}$ must be homeomorphic to $F$, and now the image of $\tau$ in
$F^{\prime}$ is naturally a triangulation $\tau^{\prime}$ of $F^{\prime}$.
(Note that even if $\tau$ is a simplicial triangulation of $F$, the
triangulation $\tau^{\prime}$ of $F^{\prime}$ may not be simplicial.)

Distinct edges of $\tau$ cannot map to the same edge in $F^{\prime\prime}$.
Thus if distinct edges $e_{1}$ and $e_{2}$ of $\tau$ map to the same edge in
$F^{\prime}$, it must be because certain $2$--gons in $F^{\prime\prime}$ were
collapsed, so that $l(e_{1})=l(e_{2})$. It follows that there is a well
defined function $l^{\prime}$ on the edges of $\tau^{\prime}$, given by
$l^{\prime}(e^{\prime})=l(e)$, where $e^{\prime}$ is any edge of $\tau
^{\prime}$ and $e$ is any edge of $\tau$ which maps to $e^{\prime}$. Clearly
$l^{\prime}$ is never zero and so is a simplicial metric on $\tau^{\prime}$.
Further, as $f_{0}$ factors through the simplicial map $f_{0}^{\prime
}:F^{\prime}\rightarrow M$, it follows that $E(f_{0}^{\prime})=\mathcal{I}$.
Finally, as $F^{\prime}$ is obtained from $F$ by collapsing disjoint
contractible subsets, $f_{0}^{\prime}$ must be homotopic to $f_{0}$.
\end{proof}

A consequence of Corollary~\ref{simplyessential} is that any potential
counterexample for the simple loop conjecture in a hyperbolic $3$--manifold
can be chosen to be geometrically controlled; it can be homotoped to be
simplicial with area less than $2\pi\left\vert \chi\right\vert $. Since the
simple loop conjecture has been proven for Seifert fiber spaces \cite{Hass:87}
and for graph manifolds \cite{RubinsteinWang}, the hyperbolic case is central
to this problem. Note that in \cite{CooperManning}, Cooper and Manning gave an
example of a representation of a surface group into $PSL_{2}(\mathbb{C})$
which has nontrivial kernel, but no element of the kernel is represented by an
essential simple closed curve. However this is not a counterexample to the
simple loop conjecture as their representation is not known to be discrete,
nor to have image in the fundamental group of a hyperbolic $3$--manifold.

\section{Maps between surfaces and canonical triangulations}

\label{diffeo}

In this section we consider a metrized triangulated closed surface
$(F_{1},\tau,l)$, and a closed surface $F_{2}$ with Euler characteristic
$\chi<0$, and with a Riemannian metric of non-positive curvature. We will
consider maps from $F_{1}$ to $F_{2}$ which are homotopic to a given
homeomorphism. We show that in the homotopy class of such a map there is a
simplicial harmonic map which is a homeomorphism. If $F_{2}$ is negatively
curved, this simplicial harmonic map is then unique in its homotopy class.
Thus the simplicial harmonic map never contains unnecessary singularities or
folds. Hence if $\tau$ is a fixed triangulation of $F_{1}$, this map gives a
canonical triangulation of $F_{2}$ that varies continuously as we vary the
metric $h$. Note that if $F_{1}$ and $F_{2}$ were orientable, with $F_{1}$ of
genus $3$, and $F_{2}$ of genus $2$, and if $f$ was a map of degree $1$ from
$F_{1}$ to $F_{2}$, then $f$ would be homotopic to a simplicial harmonic map
$g$, by Proposition \ref{uniqueharmonicmapsexist}, but $g$ could not be a
homeomorphism nor could it be a covering map. Thus in general simplicial
harmonic maps between surfaces must have some type of singularity. See Sampson
\cite{Sampson} for an analysis of the possible singularities in the smooth case.

In order to prove the results in this section, we need to slightly restrict
the type of triangulations we allow.

\begin{definition}
A triangulation $\tau$ of a compact surface $F$ is \emph{good} if no edge of
$\tau$ forms a null homotopic loop, and no union of two edges of $\tau$ forms
a null homotopic loop.
\end{definition}

Clearly if a triangulation of $F$ is not good and $F$ has a metric of
non-positive curvature, it is impossible to have a homeomorphism of $F$ which
sends every edge to a geodesic.

For the rest of this section, we fix a metrized triangulated surface
$(F_{1},\tau,l)$, a closed surface $(F_{2},h)$ with Euler characteristic
$\chi<0$, and with a non-positively curved Riemannian metric $h$, and a
homeomorphism $f:F_{1}\rightarrow F_{2}$. Let $I(f)$ denote the infimum of the
simplicial energies of all homeomorphisms isotopic to $f$, and let $\tau^{1}$
denote the $1$--skeleton of $\tau$. Note that we can isotope $f$ to be
piecewise geodesic on $\tau^{1}$, so that the simplicial energy is now
defined. In particular, the isotopy class of $f$ always contains
homeomorphisms whose simplicial energy is defined. We will consider an energy
minimizing sequence of homeomorphisms $\{f_{n}\}$ from $F_{1}$ to $F_{2}$ each
of which is isotopic to $f$, and will show that there is a subsequence which
yields a simplicial homeomorphism $(F_{1},\tau,l)\rightarrow(F_{2},h)$ whose
energy equals $I(f)$. The special case where $\tau$ is a $1$--vertex
triangulation and $h$ is negatively curved turns out to be much simpler and
will be treated separately after the following lemma. Note that this first
lemma does not require that $\tau$ be good, nor that the metric on $F_{2}$ be
non-positively curved.

\begin{lemma}
\label{edgesaremappedtogeodesics}Suppose that $F_{1}$ and $F_{2}$ are closed
surfaces with negative Euler number. Let $f:(F_{1},\tau,l)\rightarrow
(F_{2},h)$ be a homeomorphism from $F_{1}$, with simplicial metric $l$, to
$F_{2}$ with Riemannian metric $h$. Let $\{f_{n}\}$ be a sequence of
homeomorphisms from $F_{1}$ to $F_{2}$, each of which is isotopic to $f$ and
whose simplicial energies approach $I(f)$. Then there is a subsequence of the
$f_{n}$'s whose restrictions to $\tau^{1}$ converges to a map $f_{0}:\tau
^{1}\rightarrow F_{2}$. Further any map $f_{0}$ obtained in this way maps each
edge of $\tau$ to a point or to a geodesic arc.
\end{lemma}

\begin{proof}
As $f_{n}$ has finite simplicial energy, it must map each edge of $\tau$ to a
rectifiable arc. As the sequence of simplicial energies of the $f_{n}$'s is
convergent, it is bounded. As $F_{2}$ is compact, Ascoli's Theorem shows that
the sequence of $f_{n}$'s has a subsequence whose restriction to $\tau^{1}$
converges to a continuous function $f_{0}$. If $e$ is an edge of $\tau$, then
the sequence of lengths of the arcs $f_{n}\mid e$ is also bounded, so that
$f_{0}\mid e$ must have finite length. Thus $f_{0}$ maps each edge of $\tau$
to a rectifiable arc. Further for any extension of $f_{0}$ to all of $F_{1}$,
we have $E_{S}(f_{0})=I(f)$.

We claim that if $z$ is any point of $f_{0}(\tau^{1})$ which is not the image
of a vertex of $\tau$, then there must be a disc $D$ centered on $z$ such that
$f_{0}(\tau^{1})$ meets $D$ in geodesic arcs. Assuming this claim, the fact
that there can be only finitely many images of vertices of $\tau$ immediately
implies that $f_{0}$ must map each edge of $\tau$ to a point or to a geodesic
arc, as required.

To prove our claim, suppose that $z$ is a point of $f_{0}(\tau^{1})$ which is
not the image of a vertex of $\tau$. Then suitably small neighborhoods of $z$
will also not contain the image of a vertex of $\tau$. As $f_{0}$ maps each
edge of $\tau$ to a rectifiable arc, it follows that there is a small convex
disc $D$ centered on $z$ such that $f_{0}^{-1}(\partial D)$ consists of a
finite number of points $w_{1},\ldots,w_{k}$. As $D$ does not contain the
image of a vertex, the intersection of the image of $f_{0}(\tau^{1})$ with $D$
consists of isolated points in $\partial D$ together with arcs joining the
images of the remaining $w_{i}$'s in pairs. We replace these arcs by geodesic
arcs joining the images of the $w_{i}$'s in pairs in the same way, to obtain a
new map $f_{0}^{\prime}$ of $\tau^{1}$ into $F_{2}$. As $f_{0}$ is the limit
of the homeomorphisms $\{f_{n}\}$, we could make the analogous changes to
$f_{n}$, for all large enough values of $n$. The geodesic arcs obtained when
we do this for $f_{n}$ will all be disjoint. We conclude that there is a new
sequence $f_{n}^{\prime}$ of homeomorphisms such that the limit of
$f_{n}^{\prime}\mid\tau^{1}$ is $f_{0}^{\prime}$. As no edge of $\tau$ is
lengthened by this construction, we have $E_{S}(f_{n}^{\prime})\leq
E_{S}(f_{n})$, for each $n.$ If there is any edge of $\tau$ whose image under
$f_{0}$ meets $D$ in a non-geodesic arc, our construction of $f_{0}^{\prime}$
will strictly shorten the image of this edge, so that the limit of the
sequence $\{E_{S}(f_{n}^{\prime})\}$ is strictly less than $I(f)$, which is a
contradiction. We deduce that $f_{0}(\tau^{1})$ must meet $D$ in geodesic
arcs, which completes the proof of the claim.
\end{proof}

Before proceeding further, we consider the special case when $\tau$ is a
$1$--vertex triangulation of $F_{1}$, and $F_{2}$ has a Riemannian metric of
negative curvature.

\begin{theorem}
\label{1-vertextriangulation}Let $F_{1}$ and $F_{2}$ be closed surfaces with
negative Euler number. Let $\tau$ be a good $1$--vertex triangulation of
$F_{1}$, let $h$ be a negatively curved Riemannian metric on $F_{2}$, and let
$f:(F_{1},\tau,l)\rightarrow(F_{2},h)$ be a homeomorphism. Then the following hold:

\begin{enumerate}
\item $f$ is isotopic to a simplicial homeomorphism.

\item The unique least energy simplicial map homotopic to $f$ is a simplicial
homeomorphism and is isotopic to $f$.

\item Any homeomorphism from $F_{1}$ to $F_{2}$ which is homotopic to $f$ must
be isotopic to $f$.
\end{enumerate}
\end{theorem}

\begin{remark}
By taking $F_{1}$ equal to $F_{2}$ in part 3), we see that we have a new proof
of the classical fact that homotopic homeomorphisms of $F$ are isotopic, so
long as $F\ $has negative Euler number.
\end{remark}

\begin{proof}
1) Lemma \ref{edgesaremappedtogeodesics} tells us that there is a sequence
$\{f_{n}\}$ of homeomorphisms from $F_{1}$ to $F_{2}$ each of which is
isotopic to $f$, and whose simplicial energies approach $I(f)$, such that the
restrictions of the $f_{n}$'s to $\tau^{1}$ converge to a map $f_{0}$ which
sends each edge of $\tau$ to a point or to a geodesic arc. Let $v$ denote the
vertex of $\tau$. As $\tau$ is a good $1$--vertex triangulation, $f_{0}$
cannot send an edge of $\tau$ to a point, and it must send distinct edges of
$\tau$ to distinct geodesics in $F_{2}$. In particular, the restriction of
$f_{0}$ to each edge of $\tau$ is an immersion. If $f_{0}$ fails to embed
$\tau^{1}$ in $F_{2}$, there must be an edge $e$ of $\tau$ with an interior
point $w$ such that $f_{0}(w)$ equals $f_{0}(v)$. See Figure~\ref{1vertex}. If
$U$ is a suitably small neighborhood of $f_{0}(v)$ in $F_{2}$, the image of
$\tau^{1}$ in $U$ consists of the image of a neighborhood $W$ of $w$ in $e$,
together with the image of a neighborhood $V$ of $v$ in $\tau^{1}$. As the
$f_{n}$'s are all homeomorphisms, $f_{0}(V)$ lies on one side of $f_{0}(W)$,
as shown in Figure~\ref{1vertex}.

\begin{figure}[ptbh]
\centering
\includegraphics[width=1in]{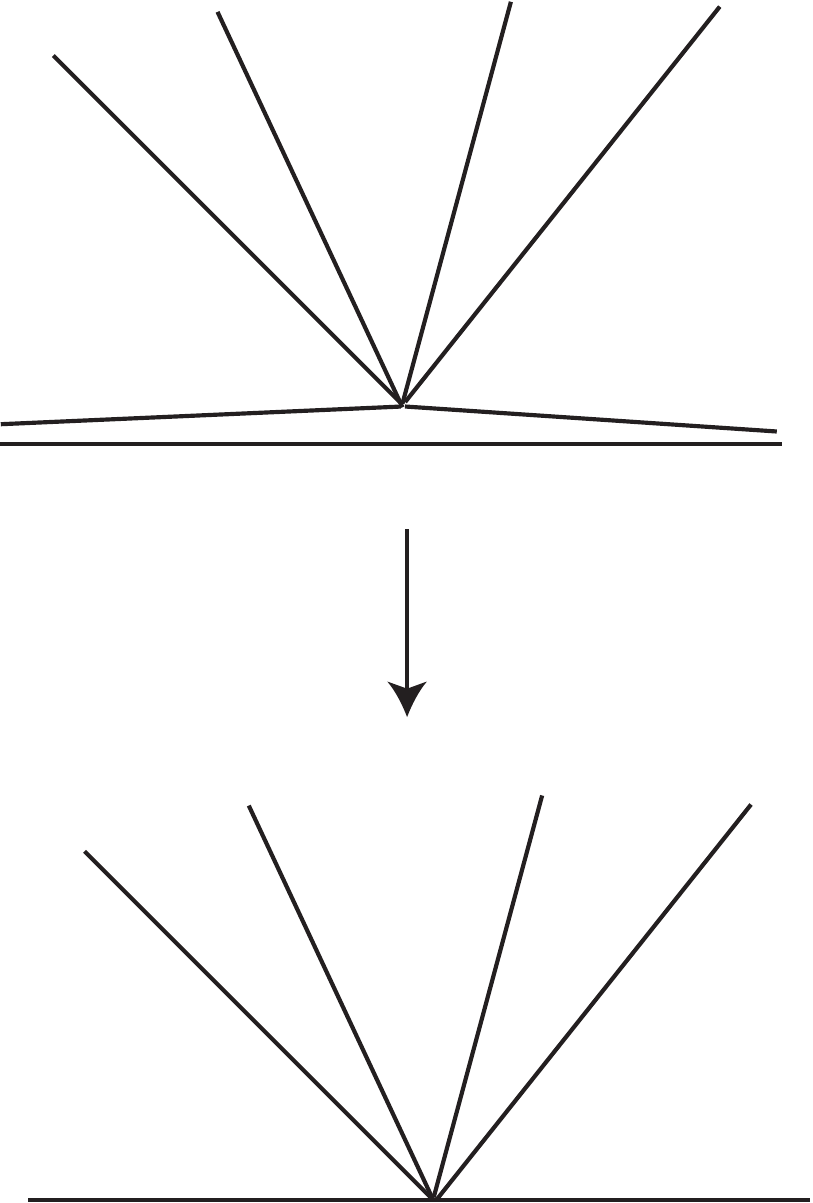} \caption{A vertex limiting to meet
the interior of an edge.}%
\label{1vertex}%
\end{figure}

Note that it is possible that some edges of $f_{0}(V)$ embed in $f_{0}(W)$. In
all cases, we will obtain a contradiction, implying that $f_{0}$ must embed
$\tau^{1}$ in $F_{2}$. Assuming this, as the $f_{n}$'s are all homeomorphisms,
and $f_{0}$ is the limit of the sequence $f_{n}\mid\tau^{1}$, there is a
unique extension of $f_{0}$ to a simplicial homeomorphism $g$ from
$(F_{1},\tau,l)$ to $(F_{2},h)$. The Alexander Trick applied to each triangle
then implies that $f$ is isotopic to this simplicial homeomorphism, as required.

To obtain a contradiction, we first consider the case where each edge of
$f_{0}(V)$ meets $f_{0}(W)$ only in $f_{0}(v)$. Now consider a homotopy of
$f_{0}$ which moves $f_{0}(v)$ a small distance orthogonally to $f_{0}(W)$ towards this side of
$f_{0}(W)$. Let $f_{0}^{\prime}$ denote the map so obtained. We further
choose $f_{0}^{\prime}$ to be equal to $f_{0}$ outside $V$, and so that
$f_{0}^{\prime}(V)$ consists of a union of geodesic arcs joining
$f_{0}^{\prime}(v)$ to $f_{0}(V)\cap\partial U$. This homotopy will shorten
the image of every edge of $\tau$, so that $f_{0}^{\prime}$ has strictly less
simplicial energy than $f_{0}$. As $f_{0}^{\prime}$ embeds $\tau^{1}$ in
$F_{2}$, it can be extended to a homeomorphism which is isotopic to $f$. As
the energy of $f_{0}^{\prime}$ is less than $I(f)$, this is the required contradiction.

Now suppose that some edges of $f_{0}(V)$ embed in $f_{0}(W)$. This implies
that a triangle of which $e$ is an edge must be collapsed onto $f_{0}(e)$. No
other triangles can collapse onto $f_{0}(e)$, as $f_{0}$ sends distinct edges
of $\tau$ to distinct geodesics in $F_{2}$. Now we apply the preceding
construction to obtain a map $f_{0}^{\prime}$ with strictly less simplicial
energy than $f_{0}$. Of course $f_{0}^{\prime}$ does not embed $\tau^{1}$ in
$F_{2}$, but by homotoping very slightly the edges collapsed onto $f_{0}(e)$,
we can obtain an embedding $f_{0}^{\prime\prime}$ of $\tau^{1}$ in $F_{2}$
with strictly less simplicial energy than $f_{0}$. Note that if $e^{\prime}$
is a geodesic segment in $F$, and we move one end $P$ of $e^{\prime}$ along a
geodesic orthogonal to $e^{\prime}$, then the derivative of the length of
$e^{\prime}$ with respect to the distance moved by $P$ is zero. This is why
$f_{0}^{\prime\prime}$ has strictly less simplicial energy than $f_{0}$, if we
move $f_{0}(v)$ a suitably small distance. As above this yields the required contradiction.

2) By construction, $g$ is a homeomorphism which is isotopic to $f$, and $g$
minimizes $E_{S}$ among all homeomorphisms isotopic to $f$. In particular, $g$
is a critical point for $E_{S}$ under all variations through homeomorphisms.
As any small deformation of $g$ must still be a homeomorphism, this implies
that $g$ is a critical point for $E_{S}$ under all variations. As $h$ is
negatively curved, it follows that $g$ is the unique least energy simplicial
map homotopic to $f$, as required.

3) Let $f_{1}$ denote a homeomorphism from $F_{1}$ to $F_{2}$ which is
homotopic to $f$. By parts 1) and 2), $f_{1}$ is isotopic to an energy
minimizing simplicial homeomorphism which must be $g$. In particular $f_{1}$
must be isotopic to $f$, as required.
\end{proof}

Now we return to the case of a general triangulation. The result we obtain is
similar to previous results obtained in the smooth and orientable setting.
Sampson \cite{Sampson} and Schoen and Yau \cite{SchoenYau:78} showed that
there is a least energy harmonic diffeomorphism in each homotopy class of
diffeomorphisms from a Riemannian surface to a surface of non-positive
curvature. Later it was shown by Jost and Schoen \cite{JostSchoen} that there
is a least energy harmonic diffeomorphism in each homotopy class of
diffeomorphisms between two Riemannian surfaces of genus at least $2$, even if
the target surface has some positive curvature. In \cite{CoronHelein}, Coron
and Helein showed that, in all cases, any harmonic diffeomorphism between two
Riemannian surfaces of genus at least $2$, must be energy minimizing in its
homotopy class, and is the unique energy minimizing map in that homotopy class.

\begin{theorem}
\label{harmonicsurfacehomeoexists} Suppose that $F_{1}$ and $F_{2}$ are closed
surfaces with negative Euler number. Let $f:(F_{1},\tau,l)\rightarrow
(F_{2},h)$ be a homeomorphism from $F_{1}$, with simplicial metric $l$, to
$F_{2}$ with a Riemannian metric of non-positive curvature. Further assume
that $\tau$ is a good triangulation. Then $f$ is isotopic to a simplicial
harmonic homeomorphism $g:(F_{1},\tau,l)\rightarrow(F_{2},h)$ that minimizes
$E_{S}$ among all homeomorphisms isotopic to $f$. Further, if $(F_{2},h)$ is
negatively curved, then $g$ is the unique simplicial harmonic map in the
homotopy class of $f$, and so minimizes energy in this homotopy class.
\end{theorem}

\begin{proof}
Recall that $I(f)$ denotes the infimum of the simplicial energies of all
homeomorphisms isotopic to $f$, and that $\tau^{1}$ denotes the $1$--skeleton
of $\tau$.

Lemma \ref{edgesaremappedtogeodesics} tells us that there is a sequence
$\{f_{n}\}$ of homeomorphisms from $F_{1}$ to $F_{2}$ each of which is
isotopic to $f$, and whose simplicial energies approach $I(f)$, such that the
restrictions of the $f_{n}$'s to $\tau^{1}$ converge to a map $f_{0}$ which
sends each edge of $\tau$ to a point or to a geodesic arc. In Lemma
\ref{f0isahomeomorphism} we will show that $f_{0}$ must be an embedding. As
the $f_{n}$'s are all homeomorphisms, and $f_{0}$ is the limit of the sequence
$f_{n}\mid\tau^{1}$, there is an extension of $f_{0}$ to a simplicial
homeomorphism $g$ from $(F_{1},\tau,l)$ to $(F_{2},h)$. By construction, $g$
is a homeomorphism which is isotopic to $f$, and $g$ minimizes $E_{S}$ among
all homeomorphisms isotopic to $f$. In particular, $g$ is a critical point for
$E_{S}$ under all variations through homeomorphisms. As any small deformation
of $g$ must still be a homeomorphism, this implies that $g$ is a critical
point for $E_{S}$ under all variations. Thus $g$ is simplicial harmonic. This
completes the proof of the first part of the theorem.

If $(F_{2},h)$ has negative curvature, Proposition
\ref{uniqueharmonicmapsexist} tells us that the homotopy class of $f$ contains
only one simplicial harmonic map, and that this map minimizes $E_{S}$ in the
homotopy class. It follows that $g$ is this map, which completes the proof of
the theorem subject to giving the proof of Lemma \ref{f0isahomeomorphism}.
\end{proof}

In order to show that $f_{0}$ must be an embedding, we need to analyze how
this might fail. First we consider what happens if an edge of $\tau$ is
collapsed to a point by $f_{0}$. Note that, as $\tau$ is good, this cannot
happen for an edge which is a loop. Before dealing with the general case, we
consider the following special case.

\begin{lemma}
\label{singleedgecollapsed}Suppose that $F_{1}$ and $F_{2}$ are closed
surfaces with negative Euler number. Let $f:(F_{1},\tau,l)\rightarrow
(F_{2},h)$ be a homeomorphism from $F_{1}$, with simplicial metric $l$, to
$F_{2}$ with Riemannian metric $h$, of non-positive curvature. Let $f_{0}%
:\tau^{1}\rightarrow F_{2}$ be a map which is the limit of the restriction to
$\tau^{1}$ of a sequence of homeomorphisms from $F_{1}$ to $F_{2}$, each of
which is isotopic to $f$, and whose simplicial energies approach $I(f)$.

Suppose that there is an edge $e$ of $\tau$ which is mapped to a point by
$f_{0}$, such that none of the edges incident to $e$ is mapped to a point.
Then there is a geodesic $\lambda$ through $f_{0}(e)$ such that every edge of
$\tau$ incident to $e$ has image contained in $\lambda$. Hence $f_{0}$ maps
every edge in the star of $e$ into $\lambda$.
\end{lemma}

\begin{proof}
Let $v$ and $w$ denote the vertices of $e$. Let $C$ denote the unit circle in
the tangent space to $F_{2}$ at $f_{0}(e)$. Each edge $e^{\prime}$ of $\tau$
which is incident to $e$ determines a point of $C$ which corresponds to the
tangent vector to the image of $e^{\prime}$. As $f_{0}$ is a limit of
homeomorphisms, we can cut $C$ at two points into two closed intervals
$A\ $and $B$ such that one contains the points determined by edges of $\tau$
incident to $v$, and the other contains the points determined by edges of
$\tau$ incident to $w$.

Suppose that these endpoints of $A$ and $B$ are not diametrically opposite
points of $C$. Then we can cut $C$ at a pair of diametrically opposite points
into two closed intervals $L\ $and $R$ one of which contains one of $A$ or $B$
in its interior. Without loss of generality, we can suppose that $L$ contains
$A$ in its interior.

Now consider a homotopy of $f_{0}$ which moves $f_{0}(v)$ a small distance
towards the midpoint of $L$, and does not move any other vertices of $\tau$,
and let $f_{0}^{\prime}$ denote the map so obtained. If $e^{\prime}$ is a
geodesic segment in $F$, and we move one end $P$ of $e^{\prime}$ along a
geodesic making an angle $\theta$ with $e^{\prime}$, then the derivative of
the length of $e^{\prime}$ with respect to the distance moved by $P$ is
$\cos\theta$. In particular, if $e^{\prime}$ is an edge of $\tau$ other than
$e$, which is incident to $v$, it must be shortened by such a homotopy, as
$\theta$ is not equal to $\pi/2$. The homotopy does increase the length of
$e$, but the contribution of the length of $e$ to the energy of $f_{0}%
^{\prime}$ remains zero to first order. It follows that $f_{0}^{\prime}$ has
strictly less simplicial energy than $f_{0}$, if we move $f_{0}(v)$ a suitably
small distance. Further $f_{0}^{\prime}$ is a limit of homeomorphisms
$f_{n}^{\prime}$. But this contradicts the energy minimizing property of
$f_{0}$. Hence the endpoints of $A$ and $B$ must be diametrically opposite in
$C$.

If the interior of $A$ or the interior of $B$ contains any points
corresponding to edges incident to $e$, we can again obtain a contradiction by
essentially the same argument. Note that, as we move the vertex $v$, any edge
$e^{\prime}$ of $\tau$ incident to $v$ whose tangent vector is an endpoint of
$A$ will not change in length to first order, as the angle $\theta$ is equal
to $\pi/2$ for such $e^{\prime}$. We conclude that there are only two tangent
vectors determined by edges of $\tau$ incident to $e$ and that these two
points of $C$ are diametrically opposite. If $\lambda$ denotes the geodesic
through $f_{0}(e)$ with these tangent vectors, then it follows that every edge
incident to $e$ has image contained in $\lambda$. If $e^{\prime}$ is an edge
in the star of $e$ which is not incident to $e$, then there are two edges of
the star of $e$ which are incident to $e$, and which together with $e^{\prime
}$ bound a triangle of $\tau$. Hence each vertex of $e^{\prime}$ is mapped
into $\lambda$ by $f_{0}$, and $f_{0}(e^{\prime})$ is homotopic into $\lambda$
fixing the vertices of $e^{\prime}$. Now the non-positive curvature of $F_{2}$
implies that $f_{0}$ must map $e^{\prime}$ into $\lambda$, which completes the
proof of Lemma \ref{singleedgecollapsed}.
\end{proof}

In the above result, the boundary $\partial\sigma$ of any triangle $\sigma$ in
the star of $e$ is mapped by $f_{0}$ into the geodesic $\lambda$, and
$\partial\sigma$ is not mapped to a point. Intuitively, the picture is that as
$n\rightarrow\infty$, the sequence $f_{n}$ is collapsing the entire triangle
$\sigma$ into the geodesic $\lambda$. Unfortunately, by itself the condition
that $f_{0}$ maps the boundary of $\sigma$ into $\lambda$ is not enough to
ensure that there is an extension of $f_{0}$ to a map of $\sigma$ into
$\lambda$. For example if $\lambda$ were a simple closed curve, $f_{0}$ could
send $\partial\sigma$ to an essential loop in $\lambda$, and if $\lambda$ were
singular, $f_{0}$ need not even send $\partial\sigma$ to a loop in $\lambda$.
In order to clarify the situation, we use the fact that if $\Lambda$ is a
geodesic in the universal cover of $F_{2}$ which lies above $\lambda$, then
$\Lambda$ is an embedded line. This is because of the non-positive curvature
of $F_{2}$. This leads us to the following definition.

\begin{definition}
Let $\lambda$ be a geodesic in $F_{2}$. Then a triangle $\sigma$ in $\tau$ is
$\lambda$\emph{--degenerate under }$f_{0}$ if

\begin{enumerate}
\item $f_{0}$ maps $\partial\sigma$ into $\lambda$, and

\item the restriction of $f_{0}$ to $\partial\sigma$ lifts to a map of
$\partial\sigma$ into $\Lambda$.
\end{enumerate}
\end{definition}

Note that if $\sigma$ is $\lambda$--degenerate under $f_{0}$, then $f_{0}%
\mid\partial\sigma$ extends to a map of $\sigma$ into $\lambda$. The proof of
the preceding lemma shows that if $\sigma$ is a triangle in the star of $e$,
then $\sigma$ is $\lambda$--degenerate under $f_{0}$, but $\partial\sigma$ is
not mapped to a point by $f_{0}$.

Next we consider the general case when edges adjacent to $e$ may also be
collapsed to the point $f_{0}(e)$. We consider the subcomplex of $\tau$
consisting of all edges which are mapped to this point by $f_{0}$ together
with each $2$--simplex $\sigma$ such that $f_{0}(\partial\sigma)=f_{0}(e)$. As
the image of $f_{0}$ cannot be a point, this subcomplex is not equal to $\tau
$. Let $K$ denote the component of this subcomplex which contains $e$. Thus
$K$ is not a single point. Lemma \ref{singleedgecollapsed} is a special case
of the following result, and much of the proof is similar. We proved Lemma
\ref{singleedgecollapsed} separately to clarify the argument. Note that
$f_{0}$ is only defined on the $1$--skeleton of $K$, but in this setting we
will abuse notation and denote $f_{0}(e)$ by $f_{0}(K)$.

\begin{lemma}
\label{manyedgescollapsed}Using the above notation, there is a geodesic
$\lambda$ through $f_{0}(K)$ such that every triangle in the star of $K$ is
$\lambda$--degenerate under $f_{0}$.
\end{lemma}

\begin{proof}
As $f_{0}$ is homotopic to the restriction of a homeomorphism, the inclusion
map of $K$ into $F_{1}$ must induce the trivial map of fundamental groups.
Thus if $N$ denotes a regular neighborhood of $K$, all except one of the
components of the closure of $F_{1}-N$ must be a disc, and the union of $N$
with these discs must itself be a disc. Suppose there is a disc component $B$
of the closure of $F_{1}-N$, and let $R\ $denote the component of $F_{1}-K$
which contains $B$. Then $\partial R$ is contained in $\tau^{1}\cap K$, so is
mapped to $f_{0}(K)$ by $f_{0}$. Hence for $n$ large, $\partial R$ is mapped
arbitrarily close to $f_{0}(K)$ by $f_{n}$. As $B$ is a disc and $f_{n}$ is a
homeomorphism, it follows that the same holds for any edge of $\tau^{1}$ which
is contained in $R$, so $f_{0}$ must map each such edge to $f_{0}(K)$. It
follows that $R$ itself must be contained in $K$, which is a contradiction. We
conclude that no component of the closure of $F_{1}-N$ can be a disc, so that
$N$ itself is a disc. In particular, $K$ must be simply connected. Now let
$\partial K$ denote the subcomplex in which $K$ intersects the closure of
$F_{1}-K$. As $K$ is simply connected, a vertex of $K\ $separates $K$ if and
only if it locally separates $K$. As $K$ is finite, there must be at least two
vertices of $\partial K$ which are "extreme" and so do not separate $K$. As
any such vertex does not locally separate $K$, it must have connected link in
$K$. Let $v$ and $w$ denote two such vertices of $\partial K$. Now we will
argue very much as in the proof of Lemma \ref{singleedgecollapsed}. See
Figure~\ref{Kcollapse}.

\begin{figure}[ptbh]
\centering
\includegraphics[width=4in]{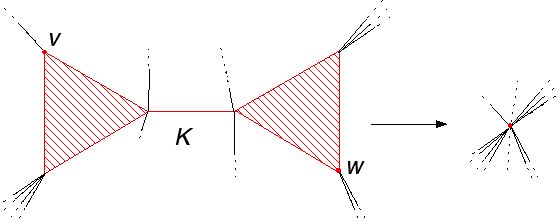} \caption{A subcomplex collapses to
a point.}%
\label{Kcollapse}%
\end{figure}

Let $C$ denote the unit circle in the tangent space to $F_{2}$ at $f_{0}(K)$.
Note that each edge of $\tau-K$ which is incident to $K$ cannot be mapped to a
point by $f_{0}$. Thus each such edge $e$ determines a point of $C$ which
corresponds to the tangent vector to the image of $e$. As $f_{0}$ is a limit
of homeomorphisms, and $v$ and $w$ have connected link in $K$, we can cut $C$
at two points into two closed intervals $A\ $and $B$ such that one contains
the points determined by edges of $\tau-K$ incident to $v$, and the other
contains the points determined by edges of $\tau-K$ incident to $w$. Suppose
that these endpoints of $A$ and $B$ are not diametrically opposite points of
$C$. Then we can cut $C$ at a pair of diametrically opposite points into two
intervals $L\ $and $R$ one of which contains one of $A$ or $B$ in its
interior. Without loss of generality, we can suppose that $L$ contains $A$ in
its interior. Now a homotopy of $f_{0}$ which moves $f_{0}(v)$ a suitably
small distance towards the midpoint of $L$, and does not move any other
vertices of $\tau$, will yield a map $f_{0}^{\prime}$ with strictly less
simplicial energy than $f_{0}$. For each edge of $\tau-K$ which is incident to
$v$ is shortened, and although edges of $K$ which are incident to $v$ increase
in length, the contribution of the length of each such edge to the energy of
$f_{0}^{\prime}$ remains zero to first order. Also all other edges are
unchanged. Further $f_{0}^{\prime}$ is a limit of homeomorphisms
$f_{n}^{\prime}$. As before, this contradicts the energy minimizing property
of $f_{0}$, so that the endpoints of $A$ and $B$ must be diametrically
opposite in $C$. But if the interior of $A$ or the interior of $B$ contains
any points corresponding to edges of $\tau-K$ incident to $v$ or $w$
respectively, we can obtain a contradiction by essentially the same argument.
Note that any edge of $\tau-K$ incident to $v$ whose tangent vector is an
endpoint of $A$ will not change in length to first order. We conclude that
there are only two tangent vectors determined by edges of $\tau-K$ incident to
$v$ or $w$ and that these two points of $C$ are diametrically opposite. If
$\lambda$ denotes the geodesic through $f_{0}(K)$ with these tangent vectors,
then it follows that every edge of $\tau-K$ incident to $v$ or to $w$ has
image contained in $\lambda$. As $f_{0}$ is a limit of homeomorphisms, it
follows that every edge of $\tau-K$ incident to $K$ has image contained in
$\lambda$. As in the proof of Lemma \ref{singleedgecollapsed}, it follows that
$f_{0}$ maps every edge in the star of $K$ into $\lambda$. It also follows
that every triangle in the star of $K$ is $\lambda$--degenerate under $f_{0}$,
which completes the proof of Lemma \ref{manyedgescollapsed}.
\end{proof}

Note that, in the above result, if $\sigma$ is a triangle in the star of $K$
which does not lie in $K$, then $\sigma$ is $\lambda$--degenerate under
$f_{0}$, but $\partial\sigma$ is not mapped to a point by $f_{0}$.

Now we simply suppose that we have a geodesic $\lambda$ in $F_{2}$ and a
triangle $\sigma$ of $\tau$ such that $\sigma$ is $\lambda$--degenerate under
$f_{0}$, but $\partial\sigma$ is not mapped to a point by $f_{0}$. We consider
the subcomplex of $\tau$ which consists of the union of all triangles which
are $\lambda$--degenerate under $f_{0}$, and let $L$ be the component of this
subcomplex which contains $\sigma$. Let $\partial L$ denote the intersection
of $L$ with the closure of $F_{1}-L$.

\begin{lemma}
\label{noedgeincidenttodLcollapses}Using the above notation, if $v$ is a point
of $\partial L$, and $e$ is an edge of $\tau$ which is incident to $v$, then
$f_{0}$ does not collapse $e$ to a point.
\end{lemma}

\begin{proof}
If $f_{0}$ collapses $e$ to a point, we consider the subcomplex of $\tau$
consisting of all edges mapped to this point together with all $2$--simplices
whose boundary is mapped to this point, and let $K$ denote the component of
this subcomplex which contains $e$. Lemma \ref{manyedgescollapsed} tells us
that $f_{0}$ maps every edge in the star of $K$ into some geodesic $\mu$. As
$\partial\sigma$ is not collapsed to a point, it follows that $K$ does not
contain some edge $E$ of $\partial\sigma$. Pick a path in $L$ which joins $K$
to $E$. Some edge $E^{\prime}$ of this path must lie in the star of $K$ but
not in $K$. This edge $E^{\prime}$ is not collapsed to a point by $f_{0}$. As
$E^{\prime}$ lies in $L$ and in the star of $K$, the image under $f_{0}$ of
$E^{\prime}$ must lie in $\lambda$ and in $\mu$. It follows that $\lambda=\mu
$. As the star of $e$ is contained in the star of $K$, it follows that every
triangle of the star of $e$ is $\lambda$--degenerate under $f_{0}$, which
contradicts the fact that $v$ lies in $\partial L$. This contradiction
completes the proof of the lemma.
\end{proof}

Of course $L$ itself need not be a surface, but it will fail to be a surface
precisely at those vertices whose link is not connected. We temporarily remove
all such vertices and let $L_{0}$ denote the closure of the component of the
resulting object which contains $\sigma$. Thus $L_{0}$ also need not be a
surface but is obtained from a connected surface $U$ by (possibly) identifying
certain vertices in its boundary. Now we pick a boundary component $C$ of $U$.
Note that any edge of $C$ also lies in $\partial L$, and so is not collapsed
to a point by $f_{0}$. In what follows we will assume that $C$ embeds in
$F_{1}$. If this is not the case, we can push $C$ slightly into the interior
of $U$ to obtain an embedded circle $C^{\prime}$ in $F_{1}$, and then apply
the following arguments using $C^{\prime}$ in place of $C$.

Let $\widetilde{F_{1}}$ and $\widetilde{F_{2}}$ denote the universal covers of
$F_{1}$ and $F_{2}$ respectively, where $\widetilde{F_{1}}$ has the
triangulation $\widetilde{\tau}$ induced from $\tau$ and $\widetilde{F_{2}}$
has the Riemannian metric induced from $h$. Let $\widetilde{\tau}^{1}$ denote
the $1$--skeleton of $\widetilde{\tau}$, let $\widetilde{f_{0}}:\widetilde
{\tau}^{1}\rightarrow\widetilde{F_{2}}$ be a map which covers $f_{0}$, and let
$\Lambda$ be a geodesic in $\widetilde{F_{2}}$ above $\lambda$. Recall that,
as $F_{2}$ is non-positively curved, $\Lambda$ is an embedded line. Let
$\widetilde{C}$ denote a component of the pre-image in $\widetilde{F_{1}}$ of
$C$, chosen so that $\widetilde{f_{0}}(\widetilde{C})\subset\Lambda$, and let
$\widetilde{U}$ denote the component of the pre-image in $\widetilde{F_{1}}$
of $U$ which has $\widetilde{C}$ as a boundary component. Note that
$\widetilde{C}$ is a simple closed curve or a line embedded in $\widetilde
{F_{1}}$. As each triangle in $L$ is $\lambda$--degenerate under $f_{0}$, it
follows that the restriction of $\widetilde{f_{0}}$ to $\widetilde{C}$ maps
$\widetilde{C}$ into $\Lambda$.

\begin{lemma}
\label{Chas0or2criticalpoints} Using the above notation, one of the following holds:

\begin{enumerate}
\item The map $\widetilde{f_{0}}:\widetilde{C}\rightarrow\Lambda$ is a
homeomorphism, or

\item $\widetilde{C}$ is a circle, and the map $\widetilde{f_{0}}%
:\widetilde{C}\rightarrow\Lambda$ has precisely two critical points.
\end{enumerate}
\end{lemma}

\begin{proof}
To prove this, choose a homeomorphism $f_{n}$, for large $n$, lift to the
universal cover and consider the induced map from $\widetilde{C}$ to
$\widetilde{F_{2}}$. This map embeds $\widetilde{C}$ very close to $\Lambda$,
and edges of $\widetilde{C}$ are almost geodesic. At each vertex $v$ of
$\widetilde{C}$, either $\widetilde{C}$ is almost straight, or it turns
through an angle of approximately $\pi$. The vertices of $\widetilde{C}$ where
the turn angle is approximately $\pi$ correspond to the critical points of
$\widetilde{f_{0}}$. At such a vertex, $\widetilde{U}$ must lie on the side of
$\widetilde{C}$ where the internal angle is almost zero. For any triangle of
$\widetilde{\tau}$ which is incident to $v$ and on that side of $\widetilde
{C}$ must be mapped to $\Lambda$ by $\widetilde{f_{0}}$ and hence must lie in
$\widetilde{U}$.

If $\widetilde{C}$ is a simple closed curve, we know it must have total
curvature close to $0$ or to $\pm2\pi$. Now each sharp bend contributes
essentially $\pm\pi$, and every other vertex and every edge contributes
essentially zero. Further the contributions of the sharp bends all have the
same sign, as the side of $\widetilde{C}$ where the internal angle is almost
zero always lies in $\widetilde{U}$. It follows immediately that
$\widetilde{f_{0}}$ has exactly zero or two critical points. In the first
case, $\widetilde{f_{0}}$ would map $\widetilde{C}$ to $\Lambda$ by an
immersion which is impossible. Thus part 2) of the lemma must hold.

If $\widetilde{C}$ is a line, it must be properly embedded, so for any
subinterval $I$ of $\Lambda$, there is a subinterval $J$ of $\widetilde{C}$
which \textquotedblleft starts at one end of $I$ and ends at the
other\textquotedblright. For such an interval $J$, the total curvature must be
close to zero. Again the contributions of the sharp bends all have the same
sign, so it follows immediately that $J$ has no critical points. Hence
$\widetilde{f_{0}}$ has no critical points, so that $\widetilde{f_{0}}$ maps
$\widetilde{C}$ to $\Lambda$ by a homeomorphism, showing that part 1) of the
lemma holds. This completes the proof of the lemma.
\end{proof}

We can now describe $U\ $and $L_{0}$ as follows. See Figure~\ref{collapse}.

\begin{lemma}
\label{Uisdiscorannulus}Using the above notation, one of the following holds:

\begin{enumerate}
\item $L_{0}$ is an annulus, $\lambda$ is a simple closed curve, and $f_{0}$
maps each component of $\partial L_{0}$ to $\lambda$ by a homeomorphism.

\item $L_{0}$ is a Moebius band, $\lambda$ is a simple closed curve, and
$f_{0}$ maps $\partial L_{0}$ to $\lambda$ by a double covering.

\item $U$ is a disc, and $f_{0}$ maps $\partial U=C$ to $\lambda$ with exactly
two critical points. Further, either $L_{0}$ is equal to $U$, or $L_{0}$ is
obtained from $U$ by identifying the two critical points of $C$. In the case
when $L_{0}$ is not equal to $U$, the geodesic $\lambda$ must be either a
simple closed curve or cross itself at the image of the critical points.
\end{enumerate}
\end{lemma}

\begin{figure}[ptbh]
\includegraphics[width=2in]{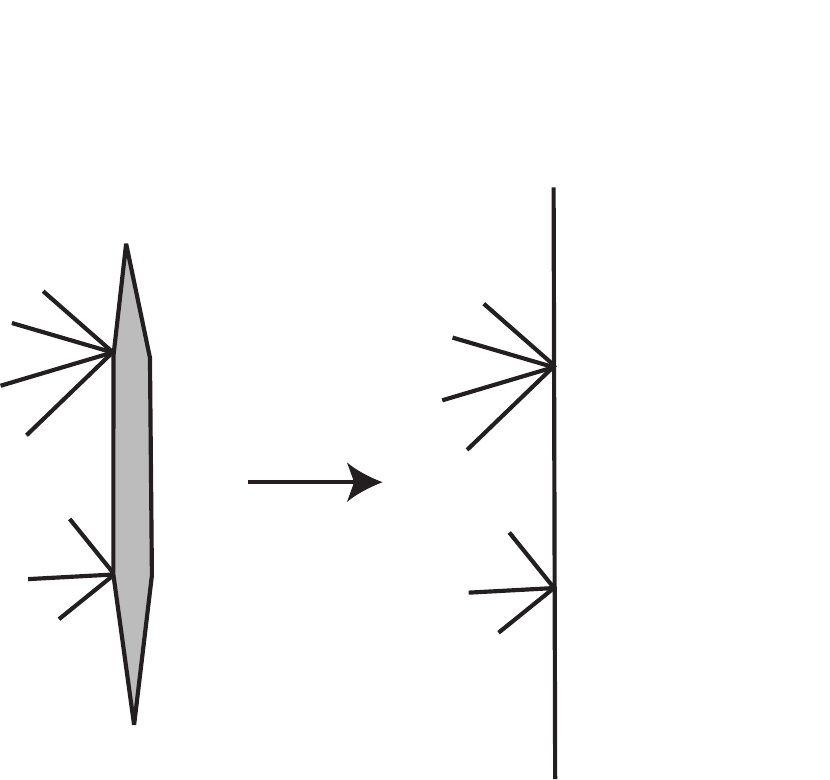}
\ \ \ \ \includegraphics[width=2in]{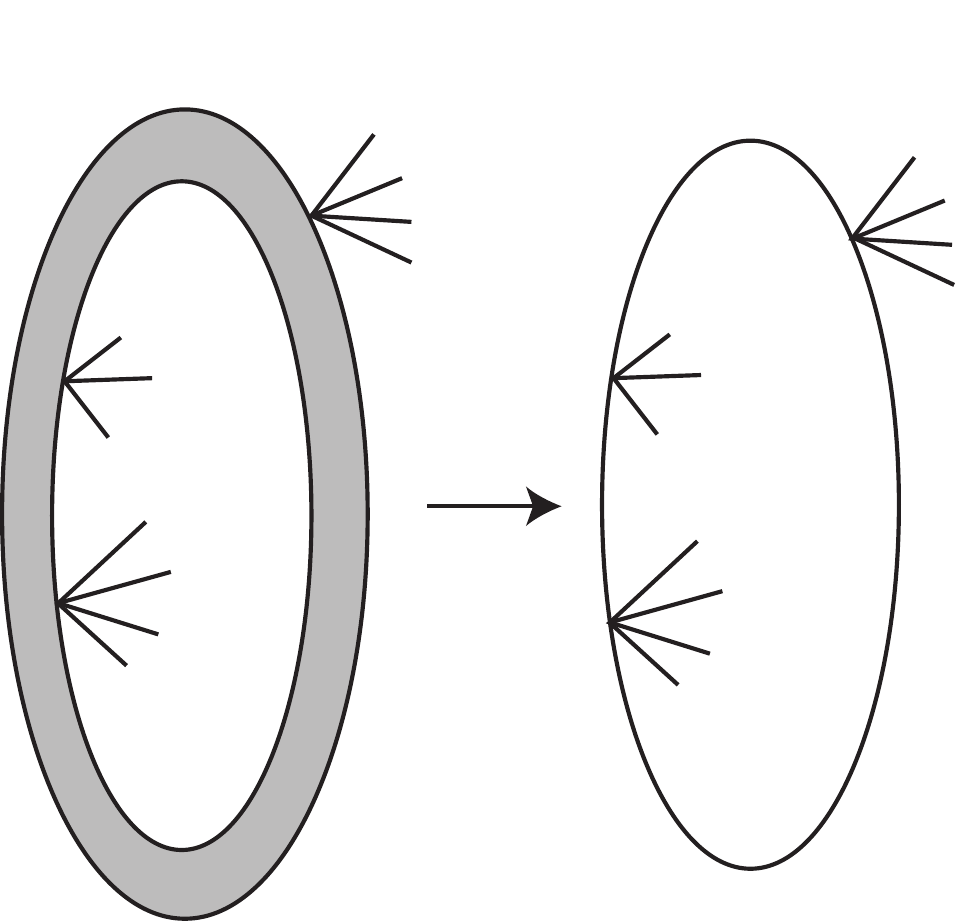} \caption{Two cases in which
$L_{0}$ collapses to a geodesic}%
\label{collapse}%
\end{figure}

\begin{proof}
We apply Lemma \ref{Chas0or2criticalpoints} to each boundary component of $U$.

Suppose that some boundary component $C$ of $U$ is such that the map
$\widetilde{f_{0}}:\widetilde{C}\rightarrow\Lambda$ is a homeomorphism. As $C$
is a simple closed curve, it follows that $\lambda$ is a simple closed curve
in $F_{2}$, and that $f_{0}\mid C:C\rightarrow\lambda$ is a covering map. As
$C$ is embedded in $F_{1}$, it follows that $f_{0}\mid C:C\rightarrow\lambda$
must be a homeomorphism or a double covering. Further in the double covering
case, $C$ must bound a Moebius band in $F_{1}$. It follows that the inclusion
of $U$ into $F_{1}$ maps $\pi_{1}(U)$ to an infinite cyclic subgroup of
$\pi_{1}(\lambda)$, and hence that $U$ must be homeomorphic to an annulus with
some discs removed, or to a Moebius band with some discs removed. As in the
proof of Lemma \ref{manyedgescollapsed}, it follows that no discs are removed
so that $U$ is an annulus or Moebius band. Now it follows that $L_{0}$ must
equal $U$ as it is not possible to have any points of $\partial U$ identified
in $F_{1}$. Thus we have cases 1) or 2) of the lemma.

Now suppose that every boundary component of $U$ maps to $\lambda$ with
exactly two critical points. We will show that we have case 3). Each boundary
component of $U$ is mapped to a null homotopic loop in $F_{2}$, and so must be
null homotopic in $F_{1}$. Hence each boundary component of $U$ bounds a disc
in $F_{1}$. As in the proof of Lemma \ref{manyedgescollapsed}, it follows that
$U$ is a disc. Further, it follows that $L_{0}$ is equal to $U$, or $L_{0}$ is
obtained from $U$ by identifying the two critical points of $C$, as required.
When $L_{0}$ is not equal to $U$, the image of $L_{0}$ in $\lambda$ is a
simple closed curve, so that either $\lambda$ is a simple closed curve or
$\lambda$ crosses itself at the image of the critical points. This completes
the proof that we have case 2).
\end{proof}

Now we are ready to prove that $f_{0}$ must be an embedding.

\begin{lemma}
\label{f0isahomeomorphism} Suppose that $F_{1}$ and $F_{2}$ are closed
surfaces with negative Euler number. Let $f:(F_{1},\tau,l)\rightarrow
(F_{2},h)$ be a homeomorphism from $F_{1}$, with simplicial metric $l$, to
$F_{2}$ with a Riemannian metric of non-positive curvature. Further assume
that $\tau$ is a good triangulation. Let $f_{0}:\tau^{1}\rightarrow F_{2}$ be
a map which is the limit of the restriction to $\tau^{1}$ of a sequence of
homeomorphisms from $F_{1}$ to $F_{2}$ each of which is isotopic to $f$, and
whose simplicial energies approach $I(f)$. Then $f_{0}$ is an embedding.
\end{lemma}

\begin{proof}
Recall from Lemma \ref{edgesaremappedtogeodesics} that $f_{0}$ maps each edge
of $\tau$ to a point or to a geodesic arc. Suppose that $f_{0}$ is not an
embedding. As $f_{0}$ is a limit of homeomorphisms, it must fail to embed the
boundary of some triangle of $\tau$. Thus one of the following cases occurs.

\begin{enumerate}
\item Some vertex is mapped to the image of the interior of an immersed edge
of $\tau$,

\item some edge of $\tau$ is collapsed to a point by $f_{0}$, or

\item some triangle of $\tau$ is mapped to a geodesic segment, not a point.
\end{enumerate}

In case 1), there is a vertex $v$ and an immersed edge $e$ of $\tau$ with an
interior point $w$ such that $f_{0}(w)$ equals $f_{0}(v)$. The proof of part
1) of Theorem \ref{1-vertextriangulation} provides a contradiction if no edge
of $\tau$ incident to $v$ is mapped into $f_{0}(e)$ or to a point. See
Figure~\ref{1vertex}. It follows that if case 1) occurs, we must also be in
case 2) or case 3).

In case 2), if an edge $e$ of $\tau$ is collapsed to a point by $f_{0}$, we
define the subcomplex $K$ as we did just before Lemma \ref{manyedgescollapsed}%
. That lemma implies that there is a geodesic $\lambda$ such that if $\sigma$
is a triangle in the star of $K$ which does not lie in $K$, then $\sigma$ is
$\lambda$--degenerate under $f_{0}$, but $\partial\sigma$ is not mapped to a
point by $f_{0}$. Thus in all cases, there is a geodesic $\lambda$ in $F_{2}$
and a triangle $\sigma$ of $\tau$ such that $\partial\sigma$ is not collapsed
to a point by $f_{0}$, and $\sigma$ is $\lambda$--degenerate under $f_{0}$.

As in the preceding three lemmas, we define the subcomplexes $L$ and $L_{0}$
of $F_{1}$, and the surface $U$. Lemma \ref{Uisdiscorannulus} tells us that
$U$ is an annulus, Moebius band or disc. In each case, we claim there is at
least one vertex $v$ of $\partial U$ which is not a critical point of the map
$f_{0}\mid\partial U:\partial U\rightarrow\lambda$. This claim is trivial in
the first two cases. In the third case when $U$ is a disc, our hypothesis that
$\tau$ is a good triangulation implies that $\partial U$ must contain at least
three vertices, so that at least one of these vertices is not a critical
point, as claimed.

Let $v$ be any vertex of a component $C$ of $\partial U$ which is not a
critical point, and let $star(v)$ denote the star of $v$ in $F_{1}$. As $C$ is
a boundary component of $U$, the two edges of $C$ incident to $v$ together
divide $star(v)$ into two discs, one of which is contained in $U$, and hence
is contained in $L_{0}$. We denote this disc by $D$, and let $D^{\prime}$
denote the other disc. Thus $star(v)$ equals $D\cup D^{\prime}$, and $D\cap
D^{\prime}$ equals two edges of $C$. We claim that no edge of $\tau-D$ which
is incident to $v$ can lie in $L$. Let $e$ be an edge of $\tau-D$ which is
incident to $v$. Thus $e$ splits $D^{\prime}$ into two subdiscs $D_{1}$ and
$D_{2}$. Now suppose that $e$ lies in $L$, so that $f_{0}(e)\subset\lambda$.
Recall from Lemma \ref{noedgeincidenttodLcollapses}, that $f_{0}$ does not
collapse $e$ to a point. As $f_{0}$ is a limit of homeomorphisms, it follows
that $f_{0}$ maps all the edges in one of $D_{1}$ or $D_{2}$ into $\lambda$.
If $f_{0}(D_{i})\subset\lambda$, then $D_{i}$ must be contained in $L$. Hence
$D\cup D_{i}$ is contained in $L$. But this implies that the interior of an
edge of $C$ lies in the interior of $L$ which contradicts the definition of
$C$. It follows that no edge of $\tau-D$ which is incident to $v$ can lie in
$L$, as claimed. Recall again that Lemma \ref{noedgeincidenttodLcollapses}
shows that every edge incident to $v$ is not mapped to a point by $f_{0}$. As
$f_{0}$ is a limit of homeomorphisms, all the edges of $\tau-D$ which are
incident to $v$ must be mapped by $f_{0}$ to the same side of $\lambda$. Note
that this statement makes crucial use of the fact that $v$ is not a critical
point of the map from $\partial U$ to $\lambda$.

Now consider the sequence of homeomorphisms $f_{n}$ whose restriction to
$\tau^{1}$ converges to $f_{0}$. As $n\rightarrow\infty$, $L_{0}$ collapses
down to the geodesic $\lambda$. Thus every edge of $L_{0}$ either collapses to
a point or to a geodesic segment along $\lambda$. Consider a small homotopy of
$f_{0}$ that only moves vertices of $L_{0}$, moves these vertices in a
direction orthogonal to $\lambda$, and does not move the critical points of
$\partial U$, if any. The contribution to the energy from the edges of $L_{0}$
is constant to first order for sufficiently small such deformations. If such a
deformation moves points of $\partial L_{0}$ away from $L_{0}$, then it
strictly shortens each edge of $\tau$ that does not lie in $L_{0}$ and is
incident to $L_{0}$ at a noncritical point of $\partial L_{0}$. The
deformation will therefore shorten these edges at a rate bounded below by a
linear function of the distance moved. If $\partial U$ has no critical points,
we conclude that for large enough values of $n$, isotoping $f_{n}$ by moving a
component of $\partial L_{0}$ away from $L_{0}$ yields a homeomorphism
$f_{n}^{\prime}$ whose simplicial energy is strictly smaller than that of
$f_{0}$. This contradicts the energy minimizing property of $f_{0}$. In the
remaining cases when $U$ is a disk and $\partial U$ has two critical points,
they divide $\partial U$ into two intervals, and we apply the same argument to
an isotopy of $f_{n}$ which moves one of these intervals while fixing the
endpoints. This contradiction completes the proof of Lemma
\ref{f0isahomeomorphism}.
\end{proof}

We will end this section by discussing an application of the preceding
arguments. But first we note that the arguments of this section work perfectly
well if $F_{1}$ and $F_{2}$ have non-empty boundary, and in this case, one
need not insist that the metrics be strictly negatively curved to obtain
uniqueness. The result we obtain is the following.

\begin{theorem}
\label{harmonicsurfacehomeosexistrelboundary}Suppose that $F_{1}$ and $F_{2}$
are compact surfaces with non-empty boundary. Let $f:(F_{1},\tau
,l)\rightarrow(F_{2},h)$ be a homeomorphism from $F_{1}$, with simplicial
metric $l$, to $F_{2}$ with a non-positively curved Riemannian metric $h$ such
that $\partial F_{2}$ is locally convex. Further assume that $\tau$ is a good
triangulation. Then $f$ is isotopic rel $\partial F_{1}$ to a simplicial
harmonic homeomorphism $g:(F_{1},\tau,l)\rightarrow(F_{2},h)$ that minimizes
$E_{S}$ among all homeomorphisms isotopic to $f$ rel $\partial F_{1}$.
Further, $g$ is the unique simplicial harmonic map in the homotopy class of
$f$, and so minimizes energy in this homotopy class.
\end{theorem}

Now let $F$ be a compact surface with a Riemannian metric $h$, and let $\tau$
denote a good triangulation of $F$. We will say that a triangulation of $F$ is
\textit{straight} if each edge is a geodesic arc. We will be interested in the
number and homotopy type of the components of the space of straight
triangulations of $F$ of a fixed combinatorial type. Note that Theorems
\ref{harmonicsurfacehomeoexists} and
\ref{harmonicsurfacehomeosexistrelboundary} tell us that for any good
triangulation $\tau$ of $F$, the space of straight triangulations of $F$
modelled on $\tau$ is non-empty. We will show that, with certain curvature
restrictions, there is a natural bijection between the components of this
space and the group of isotopy classes of homeomorphisms of $F$. When $F$ is
the $2$--disk with a flat metric, Bloch, Connelly and Henderson
\cite{BlochConnellyHenderson} proved this, and they proved, in addition, that
the components of this space are homeomorphic to some Euclidean space. In
related work, Awartani and Henderson \cite{AwartaniHenderson} considered the
problem of showing that when $F$ is the $2$--sphere, the analogous space has
the homotopy type of $O(3)$. And Bloch \cite{Bloch} showed that the space of
embeddings with geodesic edges and convex image of a triangulated $2$--disk
into the plane has the homotopy type of $O(2)$.

If we have a simplicial metric $l$ on $\tau$, and if $g$ is any simplicial
homeomorphism from $(F,\tau,l)$ to $(F,h)$, then $g(\tau)$ is a straight
triangulation of $F$. Conversely a straight triangulation of $(F,h)$ by $\tau$
determines a unique simplicial homeomorphism from $(F,\tau,l)$ to $(F,h)$.
Thus, after fixing $\tau$ and its simplicial metric, we can identify the space
of straight triangulations of $F$ by $\tau$ with the space of simplicial
homeomorphisms of $F$. Note that any triangulation $\tau$ of $F$ admits a
simplicial metric $l$ by specifying that $l$ takes the value $1$ on every edge
of $\tau$.

Having fixed a good triangulation $\tau$ of $F$ and its simplicial metric $l$,
let $H(\tau)$ denote the space of simplicial homeomorphisms from $(F,\tau,l)$
to $(F,h)$. This is a subset of the space $Homeo(F)$ of all homeomorphisms
from $F$ to $F$. The components of $Homeo(F)$ naturally correspond to the
isotopy classes of homeomorphisms of $F$.

\begin{theorem}
\label{straighttriangulationsarestraightisotopic} Let $F$ be a compact surface
with a Riemannian metric $h$ of non-positive curvature such that $\partial F$
is locally convex. If $F$ is closed we suppose that $h$ has negative
curvature. Let $\tau$ denote a good triangulation of $F$, with simplicial
metric $l$. Using the above notation, each component of $Homeo(F)$ contains
exactly one component of $H(\tau)$.
\end{theorem}

\begin{proof}
Theorems \ref{harmonicsurfacehomeoexists} and
\ref{harmonicsurfacehomeosexistrelboundary} imply that any homeomorphism
$f:(F,\tau,l)\rightarrow(F,h)$ is isotopic rel $\partial F$ to a unique
simplicial harmonic homeomorphism $g:(F,\tau,l)\rightarrow(F,h)$ that
minimizes $E_{S}$ among all homeomorphisms isotopic to $f$. In particular,
each component of $Homeo(F)$ contains a point of $H(\tau)$. Note that the fact
that $H(\tau)$ is non-empty is already nontrivial to prove.

Next we concentrate on homeomorphisms isotopic to $f$. Let $H$ denote a
component of the space of simplicial homeomorphisms from $(F,\tau,l)$ to
$(F,h)$ which are isotopic to $f$. We will show that $H$ contains the energy
minimizing simplicial harmonic homeomorphism $g$. As this holds for all such
components, it will follow that the space of simplicial homeomorphisms from
$(F,\tau,l)$ to $(F,h)$ which are isotopic to $f$ must be connected, which
will prove the first part of the theorem.

Now $H$ is contained in the space, $Map(f)$, of all simplicial maps from
$(F,\tau,l)$ to $(F,h)$ which are homotopic to $f$. The gradient flow of the
energy functional on $Map(f)$ yields a deformation retraction of $Map(f)$ to
the unique energy minimizer $g$, which we know to be a homeomorphism. It would
be extremely convenient if this flow induced a deformation retraction of $H$
to this same energy minimizer. One would need to show that if one starts with
a simplicial homeomorphism from $(F,\tau,l)$ to $(F,h)$, then the result of
following the gradient flow remains a homeomorphism for all time.
Unfortunately this is not the case. See the end of this section for a
discussion and example. This means that we are unable to show that the
components of $H(\tau)$ are contractible, although we believe this to be the
case. Instead we consider the closure $\overline{H}$ of $H$ in $Map(f)$. As
$H$ is open in $Map(f)$, it follows that $\overline{H}-H$ is closed in
$Map(f)$. This implies that the restriction of the energy functional to
$\overline{H}-H$ attains a minimum. To see this, recall that the proof of
Proposition \ref{uniqueharmonicmapsexist} showed that given any constant $K$,
the subspace of $Map(f)$ consisting of maps with energy $\leq K$ is compact.
Thus the same holds for the intersection of this subspace with $\overline
{H}-H$. Now let $g_{0}$ be a point of $\overline{H}-H$ of least energy. Thus
$g_{0}$ is not a homeomorphism, and is a limit of simplicial homeomorphisms
$f_{n}$ in $H$. We make the following claim.

\begin{claim}
There is $n$ such that we can isotope $f_{n}$ to a simplicial homeomorphism
$f_{n}^{\prime}$ whose energy is strictly less than that of $g_{0}$.
\end{claim}

Of course, such $f_{n}^{\prime}$ also lies in $H$. Assuming this claim,
consider the gradient flow starting at $f_{n}^{\prime}$. Any map in this flow
has energy less than that of $f_{n}^{\prime}$, and so less than that of
$g_{0}$. Hence no map in this flow can lie in $\overline{H}-H$. Hence the flow
stays in $H$ and so determines an isotopy of $f_{n}^{\prime}$ to the unique
energy minimizer $g$. In particular it follows that $H$ contains $g$, which
completes the proof of the theorem.

It remains to prove the above claim. We would like to prove this in much the
same way as we proved Lemma \ref{f0isahomeomorphism}. Specifically we want to
use Lemmas \ref{manyedgescollapsed}, \ref{noedgeincidenttodLcollapses},
\ref{Chas0or2criticalpoints}, \ref{Uisdiscorannulus} and
\ref{f0isahomeomorphism}. As stated, these lemmas do not apply to our
situation, as $g_{0}$ is not an energy minimizer for all homeomorphisms. But
we can apply the arguments in these lemmas as we now describe.

Recall that all these lemmas consider a map $f_{0}:\tau^{1}\rightarrow F_{2}$
which is the limit of the restriction to $\tau^{1}$ of a sequence of
homeomorphisms from $F_{1}$ to $F_{2}$, each of which is isotopic to $f$, and
whose simplicial energies approach $I(f)$. Lemma
\ref{edgesaremappedtogeodesics} shows that any such map sends each edge of
$\tau$ to a point or a geodesic segment. In the present situation, $g_{0}$ is
a limit of simplicial homeomorphisms, so it is immediate that it sends each
edge of $\tau$ to a point or a geodesic segment, but it need not have the
minimal energy $I(f)$.

Lemma \ref{manyedgescollapsed} used the least energy assumption on $f_{0}$ to
show that if $f_{0}$ collapsed some edge $e$ to a point, then it must map the
star of some complex $K$ which contains $e$ into some geodesic $\lambda$. Now
suppose that $g_{0}$ collapses some edge $e$ to a point. The arguments in the
proof of Lemma \ref{manyedgescollapsed} show that either we get the same
conclusion for $g_{0}$, or we can isotope some $f_{n}$ to a simplicial
homeomorphism $f_{n}^{\prime}$ whose energy is strictly less than that of
$g_{0}$. An alternative way of putting this is to say that either the claim
holds or the result of Lemma \ref{manyedgescollapsed} holds for $g_{0}$.

Lemmas \ref{noedgeincidenttodLcollapses}, \ref{Chas0or2criticalpoints} and
\ref{Uisdiscorannulus} depend on Lemma \ref{manyedgescollapsed} but do not
otherwise use the energy minimizing assumption on $f_{0}$. Thus again we have
that either the claim holds or their results also hold for $g_{0}$.

Finally the arguments in the proof of Lemma \ref{f0isahomeomorphism} use the
conclusions of Lemmas \ref{noedgeincidenttodLcollapses},
\ref{Chas0or2criticalpoints} and \ref{Uisdiscorannulus} to show that if
$f_{0}$ is not a homeomorphism we can isotope $f_{n}$ to a simplicial
homeomorphism $f_{n}^{\prime}$ whose energy is strictly less than that of
$f_{0}$. As $g_{0}$ is not a homeomorphism, when we apply these arguments to
$g_{0}$, it follows that the claim must hold. This completes the proof of the
claim, and hence of Theorem \ref{straighttriangulationsarestraightisotopic}.
\end{proof}

We end this section by discussing the fact that the gradient flow need not
preserve the property of being a homeomorphism. Recall from section
\ref{simplicialenergy} the formula for the simplicial energy of a map. If we
consider varying a single vertex $P$, only those edges incident to $P$ can
contribute to any change in the energy. To each such edge $e$ we associate a
vector $v_{e}$ in the tangent space at $P$ so that $v_{e}$ is tangent to $e$,
has magnitude equal to the length of $e$, and points away from $P$. Suppose
that every edge of $\tau$ has length $1$ in our simplicial metric. Then a
simple calculation shows that under the gradient flow, the flow vector of $P$
is simply the sum of the vectors $v_{e}$. Thus $P$ is ``pulled'' much more by
long edges than by short ones. If a configuration like the one in
Figure~\ref{nonembedded} is part of a triangulation of a surface, it will
quickly cease to be embedded under the gradient flow.

\begin{figure}[ptbh]
\centering
\includegraphics[width=1.5in]{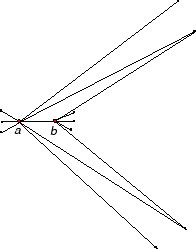} \caption{A configuration which
ceases to be embedded under the gradient flow. The vertex $a$ moves rightward
faster than vertex $b$.}%
\label{nonembedded}%
\end{figure}

\section{Higher dimensions}

\label{higherdimensions}

In the previous sections we considered maps of a triangulated closed surface
$F$ into a non-positively curved closed manifold $M$ whose dimension was
unrestricted. In this section we observe that all the preceding work can be
generalized to maps of a triangulated closed $k$--manifold $F$ into a
non-positively curved manifold $M$, for any $k\geq2$. We do this by
restricting attention to the $2$--skeleton $\tau^{(2)}$ of the triangulation
$\tau$ of $F$. This seems reasonable as the homotopy class of a map of $F$
into $M$ is determined by its restriction to $\tau^{(2)}$, as $M$ is
aspherical. In the smooth setting, much work has been done in this setting.
See \cite{EellsLemaire2}\cite{EellsLemaire3}.

For simplicity we discuss only the case when $k=3$. As in the $2$--dimensional
case, we specify a triangulation $\tau$ of the closed $3$--manifold $F$, and a
map $l$ that assigns to each edge $e_{i}$, $1\leq i\leq r$, of $\tau$ a length
$l_{i}=l(e_{i})>0$, with the lengths $l_{i}$ realizable by a Euclidean
tetrahedron for each tetrahedron of $\tau$. We call such an assignment a
\emph{simplicial metric} on $F$, and denote it by $(F,\tau,l)$, or just $l$
when the context is clear. Note that we allow the possibility that for some
triangles one of the triangle inequalities is an equality. As before we order
the vertices of $\tau$ using a $\Delta$-complex structure.

Define the \textit{simplicial $2$-volume} of $(F,\tau)$ to be the sum
$V_{2}=\Sigma l_{i}l_{j}$, where the sum is taken over all corners of all the
$2$--simplices of $\tau$.

Given a map $f:F\rightarrow M$, let $L_{i}$ denote the length of the
restriction of $f$ to the edge $e_{i}$, and as before define the $i$--th
stretch factor to be $\sigma_{i}=L_{i}/l_{i}$. The \emph{simplicial $2$-volume
of $f$} is defined to be $V_{2}(f)=\Sigma L_{i}L_{j}$, and the
\emph{simplicial $2$-energy of the map $f$} to be $E_{2}(f)=\frac{1}{2}%
\Sigma(\sigma_{i}^{2}+\sigma_{j}^{2})l_{i}l_{j}$, where the sum is taken over
all corners of all the $2$--simplices of $\tau$.

Next we need to define when a map from $F$ to $M$ is simplicial. As before we
insist that each edge of the triangulation $\tau$ is mapped to $M$ as a
geodesic arc. In order to determine $f$ canonically on the remainder of $F$,
we use the ordering of the vertices of $\tau$. For each triangle of $\tau$
this determines a cone structure, with cone point the minimal vertex of the
triangle, and as before this determines a canonical extension of $f$ to the
$2$--skeleton of $F$. We give each tetrahedron $T$ of $\tau$ the cone
structure from its minimal vertex $v$. Note that all the faces of $T$ which
meet $v$ also have cone structure with $v$ as the cone point. Thus the
canonical extension of $f$ to the tetrahedra of $\tau$ yields a well defined
simplicial map to $M$.

As in Lemma \ref{E>A}, for any simplicial metric on $F$, we have the
inequality $E_{2}(f)\geq V_{2}(f)$ with equality if and only if all the
stretch factors are equal. We say that $f$ is \textit{$2$-simplicial harmonic}
if it is a critical point of the $2$-energy functional. As in Proposition
\ref{uniqueharmonicmapsexist}, if $M$ is closed, the energy functional has a
minimum among all maps homotopic to $f$, and we can choose any minimizing map
to be simplicial. Further if $M$ is negatively curved and $f$ is nontrivial,
i.e. $f$ cannot be homotoped to have image contained in a closed geodesic,
then this minimum energy simplicial map $g$ is unique.

As in Section~\ref{families}, the fact that any map $f:F\rightarrow M$ can be
homotoped to a unique simplicial map without moving the vertices allows one to
construct families of harmonic maps with uniform area bounds.

In the preceding paragraphs, we prefixed $2$ to our definitions, because there
seem to be some natural alternative definitions. For example, we could define
the \textit{simplicial $3$-volume} of $(F,\tau)$ to be the sum $V_{3}=\Sigma
l_{i}l_{j}l_{k}$, where the sum is taken over all corners of all the
$3$--simplices of $\tau$. Then we would define the \emph{simplicial $3$-volume
of $f$} to be $V_{3}(f)=\Sigma L_{i}L_{j}L_{k}$, and the \emph{simplicial
$3$-energy of $f$} to be $E_{3}(f)=\frac{1}{2}\Sigma(\sigma_{i}^{2}+\sigma
_{j}^{2}+\sigma_{k}^{2})l_{i}l_{j}l_{k}$. In general, if the dimension of $F$
equals $d$, then one could similarly define the simplicial $k$--volume of $F$,
and the simplicial $k$--energy of $f$, for any $k$ such that $2\leq k\leq d$.
However one cannot expect there to be any connection between simplicial
$k$--volume and simplicial $k$--energy, except when $k=2$.

\end{document}